\documentclass[11pt]{amsart}
\usepackage[top=99pt, bottom=75pt, left=72pt, right=70pt]{geometry}
\usepackage[utf8]{inputenc}
\usepackage{amssymb}
\usepackage{amscd,amsxtra}
\usepackage{MnSymbol}
\usepackage{mathtools}
\usepackage{tensor}
\usepackage{faktor}
\usepackage{dsfont}
\usepackage[all]{xy}
\usepackage{graphicx}
\usepackage[rgb]{xcolor}
\usepackage{tikz, tikz-cd}
\usetikzlibrary{arrows, matrix, intersections, math, calc, decorations.pathmorphing, decorations.markings, positioning}
\usepackage{enumitem}
\usepackage{hyperref}
\definecolor{R}{RGB}{255, 0, 34}
\definecolor{B}{RGB}{0, 85, 238}
\hypersetup{colorlinks = true,
            linkcolor = R,
            urlcolor  = B,
            citecolor = B,
            anchorcolor = B}

\newtheorem{theorem}{Theorem}
\newtheorem{lemma}{Lemma}

\newtheorem{proposition}{Proposition}  
\newtheorem{corollary}{Corollary}

\numberwithin{theorem}{subsection}
\numberwithin{proposition}{subsection}
\numberwithin{lemma}{subsection}
\numberwithin{claim}{subsection}
\numberwithin{corollary}{subsection}
\numberwithin{conjecture}{subsection}
\numberwithin{definition}{subsection}
\numberwithin{remark}{subsection}

\newcommand{\beq}{\begin{equation}}
\newcommand{\eeq}{\end{equation}}
\newcommand{\beqa}{\begin{eqnarray}}
\newcommand{\eeqa}{\end{eqnarray}}
\newcommand{\beaa}{\begin{eqnarray*}}
\newcommand{\ben}{\begin{eqnarray*}}
\newcommand{\eaa}{\end{eqnarray*}}
\newcommand{\een}{\end{eqnarray*}}

\newcommand{\CC}{\mathbb{C}}
\newcommand{\ZZ}{\mathbb{Z}}
\newcommand{\PP}{\mathbb{P}}

\newcommand{\ii}{\mathbf{i}}

\newcommand{\op}{\operatorname}

\renewcommand{\O}{\mathcal{O}}

\newcommand{\F}{\mathcal{F}}
\newcommand{\M}{\mathcal{M}}

\title[K-theoretic Heisenberg algebras and KGW theory]
{K-theoretic Heisenberg algebras and permutation-equivariant
  Gromov--Witten theory}

\author{Todor Milanov }

\address{Kavli IPMU (WPI), UTIAS, The University of Tokyo, Kashiwa, Chiba 277-8583, Japan}
\email{todor.milanov@ipmu.jp}


\begin{document}

\begin{abstract}
We found an interesting application of the K-theoretic Heisenberg
algebras of Weiqiang Wang to the foundations of permutation 
equivariant K-theoretic Gromov--Witten theory. We also 
found an explicit formula for the genus 0 correlators in the
permutation equivariant Gromov--Witten theory of the point. In the
non-equivariant limit our formula reduces to a well known formula due
to Y.P. Lee. 
\end{abstract}

\maketitle

\setcounter{tocdepth}{2}
\tableofcontents

\section{Introduction}

K-theoretic Gromov--Witten (KGW) theory was introduced by Givental (see \cite{Giv2000}) and Y.P. Lee (see \cite{Lee2004}) as a generalization of the cohomological Gromov--Witten (GW) theory. One of the fundamental open problems in KGW theory is to compute the KGW invariants of the point. The genus-0 and genus-1 invariants were computed respectively in \cite{Lee1997} and \cite{LQ2014}. In the cohomological case, it is well known that the GW invariants of the point are governed by the KdV hierarchy. We expect that the KGW invariants of the point are also governed by an integrable hierarchy but so far even a conjectural description is missing. Let us point out however, that in the genus-0 case, the KGW invariants of the point are governed by a hierarchy which is a Miura transform of the dispersionless KdV hierarchy (see \cite{MT2018}). Therefore, it is quite possible that in higher-genus, the KGW invariants of the point are again solutions to the KdV hierarchy but after some complicated Miura transformation. 

There is a very interesting recent development in KGW theory which to
some extend shows that the definition of the KGW invariants must be
extended. Namely, Givental observed that in the settings of toric
geometry the so-called quantum Lefschetz hyperplane principle
fails. In particular, many natural constructions in mirror symmetry,
such as, constructing the mirror of a toric hypersurface, will fail in
KGW theory too.   In order to resolve these issues, Givental proposed an extension of KGW theory which is now called 
{\em permutation-equivariant K-theoretic Gromov--Witten} theory (see \cite{Giv2015Aug}). We believe that permutation-equivariant KGW invariants have better properties because one can use mirror symmetry to compute them. In particular, if we view the point as a hyperplane in $\PP^1$, we can construct a 2-dimensional mirror model for the permutation-equivariant KGW invariants of the point.  We would like to use this mirror model and ideas from the Eynard--Orantin recursion to compute the permutation-equivariant KGW invariants of the point. The first step in our project is to compute the genus-0 parmutation-equivariant KGW invariants of the point, that is, we would like to know the analogue of Y.P. Lee's formula \cite{Lee1997}. In some sense, this is the main motivation for this paper. Our computation is based on the techniques developed in \cite{MT2018}. Namely, we prove that the genus-0 permutation-equivariant KGW invariants are governed by an integrable hierarchy of hydrodynamic type. Using the differential equations we were able to find a closed formula for the genus-0 permutation-equivariant KGW invariants of the point (see Theorem \ref{thm:corr-0}). Let us point out that we do not work with the most general version of permutation-equivariant KGW theory, i.e., we do not allow descendants at the permutable marked points. Nevertheless, the generality that we consider should be enough for the applications that we have in mind. 

\subsection{K-theoretic Fock spaces}\label{sec:K-Fock}
The definition of permutation-equivariant KGW invariants will be recalled later on (see Section \ref{sec:pec}). From the very beginning it is clear that we are dealing with representations of the symmetric groups. However, finding a good algebraic formalism to capture the entire information is a bit tricky. We refer to \cite{Giv2015Aug} for more details on the logic used by Givental. One of the surprises in the current paper (at least to the author) is that there
is a slightly different way to organize the invariants which is equivalent to the choice made by Givental but proving the equivalence requires some non-trivial work. 

Our logic is the following. Suppose that $X$ is a smooth projective
variety.  Let $\overline{\M}_{g,n+k}(X,d)$ be the moduli space of
degree $d$ stable maps from a genus-$g$ nodal Riemann surface equipped
with $n+k$ marked points. We let the symmetric group $S_k$ act on the
moduli space by permuting the last $k$ marked points. We will also
refer to the last $k$ marked points as {\em permutable} marked
points. Let 
\ben
\op{ev}_{(n)}: \overline{\M}_{g,n+k}(X,d)\to X^n
\een
and 
\ben
\op{ev}^{(k)}:
\overline{\M}_{g,n+k}(X,d)\to X^k
\een
be the evaluation maps at respectively the first $n$ and the last $k$ marked points. The permutation-equivariant KGW invariants contain information about the $S_k$-equivariant K-theoretic pushforward $\op{ev}^{(k)}_*$. Therefore, it is natural to introduce the following graded vector
space:
\ben
\F(X)=\bigoplus_{n\geq 0} K_{S_n} (X^n),\quad K_{S_0}(X^0):=\CC,
\een
where $S_n$ is the symmetric group and $K_{S_n}(X^n)$ denotes the
$S_n$-equivariant topological K-ring with complex coefficients (see
\cite{Segal1968}). We will consider the case when $K^1(X)=0$ in order
to make the construction more transperant. The case when $K^1(X)\neq
0$ is similar. The vector space $\F(X)$ was investigated before by
Weiqiang Wang (see \cite{Wang2000}). Partially motivated by the work
of Grojnowski (see \cite{Gro1996}) and Nakajima (see \cite{Nak1997})
and following ideas of Segal, Wang proved that $\F(X)$ is a Fock
space, i.e., it is an irreducible representation of a Heisenberg
algebra. Let us outline Wang's construction with some minor modification.
The vector space $\F(X)$ is  naturally a commutative graded algebra
with multiplication defined by the induction operation (see Section
\ref{sec:ind}) 
\ben
K_{S_n}(X^n)\times
K_{S_m}(X^m)\to 
K_{S_{n+m}}(X^{m+n}),\quad
(E,F)\mapsto \op{Ind}_{S_n\times S_m}^{S_{n+m}}(E\boxtimes F).
\een
On the other hand, each graded piece $K_{S_n}(X^n)$ is a module over
the representation ring $R(S_n)$ of the symmetric group. It is well
known (see \cite{Bump2014}, Proposition 39.4) that $R(S_n)$ has a {\em
  virtual} representation $p_n$, that is, 
linear combination of irreducible representations with integer
coefficients, such that, its character is
\ben
\chi_{p_n}(g) =
\begin{cases}
  n & \mbox{ if } g=(1,2,\dots,n),\\
  0 & \mbox{ otherwise }.
\end{cases}
\een
Note that $X^n\times
p_n$ can be viewed as a virtual $S_n$-equivariant vector bundle on
$X^n$. Given $E\in K_{S_n}(X^n)$, we denote by $E\otimes p_n\in
K_{S_n}(X^n)$ the tensor product of $E$ and the trivial bundle
$X^n\times p_n$.
Let $\{\Phi_\alpha\}_{\alpha=1}^N$ be a basis of the topological
K-ring $K(X)$. Let $\{\Phi^\alpha\}_{\alpha=1}^N$ be the basis dual
to the above one with respect to the Euler pairing, that is,
$(\Phi^\alpha,\Phi_\beta)=\delta_{\alpha,\beta}$ where 
$(E,F):=\chi(E\otimes F)$. Let
$\nu_{n,\alpha}:=(\Phi^\alpha)^{\boxtimes n}\otimes p_n$.
\begin{proposition}[Wang]\label{prop:wang}
  The Fock space $\F(X)=\CC[\nu_{n,\alpha}(1\leq \alpha\leq N, n\geq
  0)]$, that is, $\nu_{n,\alpha}$ generate freely $\F(X)$ as a commutative algebra. 
\end{proposition}
We refer to \cite{Wang2000}, Proposition 3 for the proof of
Proposition \ref{prop:wang}. In fact, the result of Wang is more
general. He considered the case when $X$ is a $G$-space where $G$ is a
finite group. The corresponding Fock space $\F_G(X)$ could be viewed
as the Fock space of the orbifold $[X/G]$. From this point of view,
the work of Wang should have a generalization to the case when $X$ is
an orbifold which is not necessarily a global quotient. The orbifold
case will be important for the applications to permutation-equivariant
orbifold KGW theory.

We think of $1\in \F(X)$ as the vacuum and of multiplication by
$\nu_{n,\alpha}$ as the creation operations.
It turns out that the operators of differentiation by
$r\tfrac{\partial}{\partial \nu_{r,\alpha}}$  also have a natural
K-theoretic interpretation. Namely, the following formula holds:
\beq\label{heis_rel}
r\partial_{\nu_{r,\alpha}} (E) = \op{tr}_{(1,2,\dots,r)} \pi^{(n-r)}_*\Big(
\Phi_\alpha^{\boxtimes r} \boxtimes 1^{\boxtimes (n-r)} \otimes
\op{Res}^{S_n}_{S_r\times S_{n-r}}(E)
\Big),
\eeq
where $E\in K_{S_n}(X^n)$, $\pi^{(n-r)}: X^n=X^r\times X^{n-r}\to X^{n-r}$ is the
projection map, and $ \op{Res}^{S_n}_{S_r\times S_{n-r}}$ is the
restriction functor (see Section \ref{sec:ind}). Here the K-theoretic
pushforward $\pi^{(n-r)}_*$ yields a virtual $S_{n-r}$-bundle whose
coefficients are representations of $S_r$ and hence after taking the
trace $\op{tr}_{(1,2,\dots,r)}$ we obtain an element in
$K_{S_{n-r}}(X^{n-r})$. Let us point out that our construction of the
Fock space is slightly different from the one in
\cite{Wang2000}. Namely, the multiplication operator of Wang involves
the inverse of the Adam's operations in $K(X)$ while the differential
operator $r\partial_{\nu_{r,\alpha}}$ is represented by a contraction
operation that involves a choice in the dual vector space
$K(X)^\vee$. The proof of \eqref{heis_rel} could be obtained directly
from Wang's results but for the sake of completeness we give a
self-contained proof (see Theorem \ref{thm:Heisenberg}).

We will refer to $\F(X)$ as the K-theoretic Fock space of $X$. We will
think of $\nu=(\nu_1,\nu_2,\dots)$, where 
$\nu_r=(\nu_{r,a})_{1\leq a\leq N}$, as formal parameters. The
KGW invariants of $X$ with values in $\F(X)$ are defined as follows:
\beq\label{KGW-inv}
\langle \Phi_{a_1} L_1^{i_1},\dots, \Phi_{a_n} L_n^{i_n}
\rangle_{g,n}(\nu):=
\sum_{k=0}^\infty \sum_d Q^d \op{ev}^{(k)}_*\left(
  \O_{g,n+k,d}\otimes L_1^{i_1}\otimes \cdots \otimes L_n^{i_n}
  \op{ev}_{(n)}^*(\Phi_{a_1}\boxtimes \cdots \boxtimes \Phi_{a_n})
\right),
\eeq
where $L_i$ is the tautological line bundle formed by the cotangent
lines at the $i$-th marked point, $\O_{g,n+k,d}$ is the virtual
structure sheaf of the  moduli space $\overline{\M}_{g,n+k}(X,d)$ (see \cite{Lee2004}), and the pushforward is the $S_k$-equivariant $K$-theoretic pushforward, that is,
\ben
\op{ev}^{(k)}_*:
K_{S_k}(\overline{\M}_{g,n+k}(X,d))\to K_{S_k}(X^k)\subset
\F(X). 
\een
We will usually drop the virtual structure sheaf in the above
notation. Also, let us introduce formal parameters $t_{i,a}$ and
introduce the following formal Laurent series:
\ben
\mathbf{t}(q)=\sum_{i\in \ZZ} \sum_{a=1}^N t_{i,a} \Phi_a q^i.
\een
Then the KGW invariants with values in $\F(X)$ can be organized into
a set of formal power series in $\mathbf{t}$ of the following form:
\beq\label{nu-cor}
\langle \mathbf{t}(L_1),\dots,\mathbf{t}(L_n)\rangle_{g,n}(\nu)=
\sum_{i_1,\dots,i_n\in \ZZ}
\sum_{a_1,\dots,a_n=1}^N
t_{i_1,a_1}\cdots t_{i_n,a_n}\, 
\langle \Phi_{a_1} L_1^{i_1},\dots, \Phi_{a_n} L_n^{i_n}
\rangle_{g,n}(\nu).
\eeq
The permutation-equivariant KGW invariants of Givental (see
\cite{Giv2017}) will be recalled in Section \ref{sec:pec}. Using the
Heisenberg relations, that is, formula \eqref{heis_rel}, we will prove
in Section \ref{sec:pec} that the correlators \eqref{nu-cor} are
related to the correlators of Givental by a simple substitution (see
\eqref{substitution}). Let us point out that the above definition
\eqref{nu-cor} does not contain the entire information of
permutation-equivariant KGW theory. We did not allow descendants at
the permutable marked points. Our construction can be extended
naturally by introducing the K-theoretic Fock space $\F(X\times
\ZZ)$. We refer to Section \ref{sec:perm-desc} for more details.
However, let us point out that in this paper the most general
definition would not be needed. 

\subsection{Genus-0 integrable hierarchies}

The standard identities in KGW theory, after some minor
modifications, extend to permutation-equivariant KGW theory (see
\cite{Giv2015Oct}). More precisely, the correlators \eqref{nu-cor}
satisfy the string equation (see Section \ref{sec:string}) and if the genus is
0, then we also have the dilaton equation (see Section
\ref{sec:dilaton}) and the WDVV equations (see Section
\ref{sec:wdvv}). These identities allow us to introduce
permutation-equivariant K-theoretic quantum differential equations and
a corresponding fundamental solution called the $S$-matrix.  
Let $G$ be the $N\times N$ matrix with entries 
\beq\label{metric:G}
G_{ab}(\nu)=(\Phi_a,\Phi_b)+\langle\Phi_a,\Phi_b\rangle_{0,2}(\nu),
\eeq
where $(\Phi_a,\Phi_b)=\chi(\Phi_a\otimes \Phi_b)$ is the
Euler pairing. Let $G^{ab}(\nu)$ be the entries of the inverse matrix
$G^{-1}$. The $S$-matrix is defined by the following identity:
\beq\label{S-matrix}
G(S(\nu,q)\Phi_a,\Phi_b):= (\Phi_a,\Phi_b) +
\Big\langle \frac{\Phi_a}{1-q^{-1}L},\Phi_b
\Big\rangle_{0,2}(\nu). 
\eeq
Note that $S(\nu,q)$ is a formal power series in
$\nu$ (and the Novikov variables $Q$) whose coefficients are in
$\op{End}(K(X))(q)$. We refer to  Proposition \ref{prop:S-matrix} for
a list of properties of the $S$-matrix.
Let us define also the K-theoretic quantum cup product by 
\beq\label{K-quantum}
G(\Phi_i\bullet\Phi_j,\Phi_k):= 
\langle
\Phi_i,\Phi_j,\Phi_k
\rangle_{0,3}(\nu),\quad 1\leq i,j,k\leq N.
\eeq
Let us recall also the so-called $J$-function:
\ben
J(\nu,q):=1-q+\nu_1 +
\sum_{a=1}^N
\Phi^a\,
\Big \langle \frac{\Phi_a}{1-qL}
\Big\rangle_{0,1}(\nu).
\een
Using the string equation, it is easy to prove that $J(\nu,q)=(1-q)
S(\nu,q)^{-1} 1$.

Suppose that $v(\mathbf{t},\nu_2,\nu_3,\dots)=\sum_{\alpha=1}^N v_\alpha(\mathbf{t},\nu_2,\nu_3,\dots) \Phi_\alpha$ is a formal power series in $\mathbf{t}=(t_{k,\alpha})$ and $\nu_2,\nu_3,\dots$ with coefficients in $K(X)$. We will think of $\nu_2,\nu_3,\dots$ as parameters and quite often we suppress them in the notation. For example, we will write simply $v(\mathbf{t})$ instead of $v(\mathbf{t},\nu_2,\nu_3,\dots)$. Furthermore, we identify $\nu_1:=v(\mathbf{t})$, that is, $\nu_{1,\alpha}:=v_\alpha(\mathbf{t})$ and write $S(v(\mathbf{t}),q):= S(\nu,q)|_{\nu_1=v(\mathbf{t})}$. The following system of differential equations is the K-theoretic version of Dubrovin's principal hierarchy (see \cite{Du1996}):
\beq\label{KGW:ph}
\partial_{t_{n,\alpha}} v(\mathbf{t})= 
-\op{Res}_{q=\infty} dq (q-1)^{n-1} 
\partial v(\mathbf{t})\bullet 
S(v(\mathbf{t}),q)\Phi_\alpha
\quad 
(n\geq 0, \ 1\leq \alpha\leq N),
\eeq
where $\partial=\partial_{t_{0,1}}$ where $\Phi_1:=1\in
K(X)$. According to Milanov--Tonita (see \cite{MT2018}, Theorem 1),
the fact that $S(\nu,q)$ is a solution to the quantum differential
equations implies that the system of equations \eqref{KGW:ph} is
integrable, i.e., the system is compatible. The second goal of our
paper is to construct a solution to \eqref{KGW:ph} in terms of KGW invariants. Following Dubrovin, we will refer to this solution as the {\em topological solution}. 
The construction is the same as in \cite{MT2018}, Theorem 2.  
Put 
\ben
\mathbf{t}(q)=\sum_{k=0}^\infty \sum_{\alpha=1}^N 
t_{k,\alpha} \Phi_\alpha (q-1)^k.
\een
Note that the notation $t_{k,\alpha}$ here is slightly different from
the corresponding notation in \eqref{nu-cor}. 
\begin{theorem}\label{thm:top_sol}
Let $w(\mathbf{t},\nu_2,\nu_3,\dots)=\sum_{\alpha=1}^N w_\alpha (\mathbf{t},\nu_2,\nu_3,\dots) \Phi^\alpha$ be defined by 
\ben
\left.
w_\alpha (\mathbf{t},\nu_2,\nu_3,\dots) =
\partial_{t_{0,\alpha}}\partial_{t_{0,1}} \F^{(0)}(\mathbf{t})
\right|_{\nu_1=0}=
\sum_{n=0}^\infty 
\frac{1}{n!}\, \langle 
1,\Phi_\alpha, \mathbf{t}(L),\dots,\mathbf{t}(L)
\rangle_{0,2+n}(0,\nu_2,\nu_3,\dots).
\een
Let $v(\mathbf{t})$ be a solution to the equation
$J(v(\mathbf{t}),0)=1+w(\mathbf{t})$ where we suppressed the
dependence of $v$ and $w$ on $\nu_2,\nu_3,\dots$. Then $v(\mathbf{t})$
is a solution to the principal hierarchy \eqref{KGW:ph}.  
\end{theorem}

\subsection{Permutation-equivariant invariants of the point in genus 0}
Suppose now that $X=\op{pt}$. The Fock space in this case
$\F(\op{pt})=\bigoplus_{k=0}^\infty R(S_k)\otimes \CC=\CC[\nu_1,\nu_2,\dots]$
is the representation ring of the symmetric group where $\nu_n=p_n$
is the virtual representation of $S_n$ introduced earlier (see Section
\ref{sec:K-Fock}). 
\begin{theorem}\label{thm:corr-0}
The following formula holds:
\ben
&&
\Big\langle
\frac{1}{1-q_1L},\dots,
\frac{1}{1-q_nL},1,1
\Big\rangle_{0,n+2}(\nu)=\\
&&
\frac{1}{(1-q_1)\cdots(1-q_n)}
\Big( 1+ \frac{1}{q_1^{-1}-1} +\cdots + \frac{1}{q_n^{-1}-1}
\Big)^{n-1}\
\exp\left(
  \sum_{r=1}^\infty \frac{\nu_r}{r}
  \Big(
  1+\frac{1}{q_1^{-r}-1} +\cdots + \frac{1}{q_n^{-r}-1} \Big)
  \right)
\een
for all $n\geq 1.$
\end{theorem}

\subsection{Acknowledgements}
I am thankful to Yukinobu Toda for a very useful discussion on sheaf
cohomology and rational singularities. The idea to consider the
K-theoretic Fock space of Wang came after a talk by Timothy Logvinenko
on the MS seminar at Kavli IPMU. I am thankful to him for giving an
inspiring talk.  
This work is supported by the World Premier International Research
Center Initiative (WPI Initiative), MEXT, Japan and by JSPS Kakenhi
Grant Number JP22K03265.

\section{K-theoretic Fock space}

Let us recall the background on K-theoretic Heisenberg algebras that
will be needed in this paper. We give self-contained
proofs mostly because we would like to give the reader the chance to
grasp the techniques involved in the construction of the Fock space
$\F(X)$. We only focus on the results relevant for this paper. For
more properties and further details we refer to \cite{Wang2000}.

\subsection{Double cosets of the Young subgroups}

Suppose that $\lambda=(\lambda_1,\dots,\lambda_r)$ is a partition of $n$, that is, a decreasing sequence of integers $\lambda_1\geq \lambda_2\geq \cdots \geq \lambda_r>0$, such that, $\lambda_1+\cdots +\lambda_r=n$. Usually we write $\lambda\vdash n$ to denote that $\lambda$ is a partition of $n$. The number $\ell(\lambda):=r$ is called the {\em length} of $\lambda$. Put $\ell_k(\lambda):=\op{card}\, \{i\ |\ \lambda_i=k\}$ where $k=1,2,\dots$. Given a partition $\lambda$ we define the permutation
\ben
\sigma(\lambda):=(1,2,\dots,\lambda_1)\,
(\lambda_1+1,\lambda_1+2,\dots,\lambda_1+\lambda_2)\, \cdots\,
(\lambda_1+\cdots +\lambda_{r-1}+1,\dots, \lambda_1+\cdots +\lambda_{r-1}+\lambda_r).
\een
Note that $\{\sigma(\lambda)\}_{\lambda\vdash n}$ is a complete set of
representatives for the conjugacy classes of $S_n$ and that the
conjugacy class $C_\lambda:=\{g\sigma(\lambda)g^{-1}\ |\ g\in S_n\}$ consists of $n! z_\lambda^{-1}$ elements where $z_\lambda :=\prod_k k^{\ell_k(\lambda)} \ell_k(\lambda)!$ is the number of elements in $S_n$ that commute with $\sigma(\lambda)$. Finally, let us denote by $Y_i(\lambda)$ ($1\leq i\leq \ell(\lambda)$) the sequences $Y_1(\lambda):=\{1,2,\dots,\lambda_1\}$, $Y_2(\lambda):=
\{\lambda_1+1,\lambda_1+2,\dots,\lambda_1+\lambda_2\}$, etc. The
subgroup of $S_n$ consisting of permutations that leave the sets
$Y_i(\lambda)$ invariant for all $i$ will be denoted by
$S_\lambda$, that is,
\ben
S_\lambda:=\{\sigma\in S_n\ |\
\sigma(Y_i(\lambda))=Y_i(\lambda)\}\cong
S_{\lambda_1}\times \cdots \times S_{\lambda_r}
.
\een
The subgroups $S_\lambda$ ($\lambda\vdash n$) are
known as the {\em Young subgroups} of $S_n$.

Suppose now that $\lambda$ and $\mu$ are partitions of $n$. We would
like to recall the classification of the 
double coset classes $S_\mu\backslash S_n /S_\lambda$ (see
\cite{Zele1981}, Appendix 3, Section A3.2). For a given 
parmutation $\sigma\in S_n$, let  us define 
\ben
\gamma_{ij}:= \op{card} (\sigma (Y_i(\lambda)) \cap Y_j(\mu)),
\quad 1\leq i\leq \ell(\lambda)=:r,
\quad 1\leq j\leq \ell(\mu)=:s.
\een
Note that $\gamma_{ij}$ are non-negative integers satisfying
\beq\label{rc-constraints}
\sum_{j=1}^s \gamma_{ij} = \lambda_i,\quad
\sum_{i=1}^r \gamma_{ij}=\mu_j.
\eeq
\begin{lemma}
The  numbers $\gamma_{ij}$ defined above depend only on the double
coset of $\sigma$ in  $S_\mu\backslash S_n /S_\lambda$.
\end{lemma}
The proof follows immediately from the definition (see
\cite{Zele1981}, Appendix 3, Section A3.2). 
Suppose that we fix a matrix $\gamma=(\gamma_{ij})$ of size
$r\times s$ satisfyning the above conditions
\eqref{rc-constraints}. Let us construct a permutation $\sigma\in S_n$
such that the corresponding matrix is $\gamma$. Split each set
$Y_i(\lambda)$ into subsets $Y_{i1}(\lambda),\dots, Y_{is}(\lambda)$,
such that, $Y_{i1}(\lambda)$ consists of the first $\gamma_{i1}$
elements of $Y_i(\lambda)$, $Y_{i2}(\lambda)$ -- of the next
$\gamma_{i2}$ elements, etc. We can do this because the number of
elements in $Y_i(\lambda)$ is
$\lambda_i=\gamma_{i1}+\gamma_{i2}+\cdots + \gamma_{is}. $ Similarly,
let us split each set $Y_j(\mu)$ into subsets $Y_{1j}(\mu),\dots,
Y_{rj}(\mu)$ consisting of respectively $\gamma_{1j},\dots
\gamma_{rj}$ elements. We define $\sigma$ to be the permutation that
maps the elements of $Y_{ij}(\lambda)$ to $Y_{ij}(\mu)$ and for the
sake of definitness, let us require that $\sigma$ preserves the order
of the elements. Note that $\sigma(Y_i(\lambda))\cap
Y_j(\mu)=Y_{ij}(\mu)$ consists of $\gamma_{ij}$ elements.
\begin{lemma}
Suppose that $\gamma=(\gamma_{ij})$ is a matrix of non-negative
integers satisfying \eqref{rc-constraints}. Let $\sigma\in S_n$ be the
permutation constructed above. Then the double coset $S_\mu \sigma
S_\lambda$ consists of $\lambda! \mu!/\gamma!$ elements, where
$\lambda!=\lambda_1!\cdots \lambda_r!$,
$\mu!=\mu_1!\cdots \mu_s!$, and $\gamma!=\prod_{i,j} \gamma_{ij}!.$ 
\end{lemma}
\begin{proof}
Let us consider the map
\ben
S_\mu \times S_\lambda \to S_\mu \sigma S_\lambda,
\quad
(\tau_1,\tau_2)\mapsto \tau_1 \sigma \tau_2^{-1}.  
\een
The above formula defines a transitive action of $S_\mu\times S_\lambda$ on the double coset $S_\mu \sigma S_\lambda$. The stabilizer of $\sigma\in  S_\mu \sigma S_\lambda$ consists of pairs $(\tau_1,\tau_2)$, such that, 
$\tau_1=\sigma \tau_2 \sigma^{-1}$. Note that $\tau_1$ is uniquely
determined by $\tau_2$. We claim that $\tau_2$ leaves the subsets
$Y_{ij}(\lambda)$ invariant for all $i$ and $j$. Indeed, note that
\ben
\tau_2(Y_{ij}(\lambda)) \subset \tau_2(Y_i(\lambda))= Y_i(\lambda).
\een
On the other hand,
\ben
\sigma^{-1} \tau_1 \sigma(Y_{ij}(\lambda))=
\sigma^{-1} \tau_1 (Y_{ij}(\mu)) \subset \sigma^{-1} (Y_j(\mu)).
\een
It remains only to notice that $Y_i(\lambda)\cap \sigma^{-1}
(Y_j(\mu))=Y_{ij}(\lambda)$. Conversely, if $\tau_2$ preserves
$Y_{ij}(\lambda)$ for all $i$ and $j$, then $\tau_1=\sigma
\tau_2\sigma^{-1}$ leaves all $Y_{ij}(\mu)$ invariant. In particular,
$\tau_1\in S_\mu$. This proves that the stabilizer of $\sigma$ is
isomorphic to the subgroup of permutations in $S_n$ that leave the
subsets $Y_{ij}(\lambda)$ invariant for all $i$ and $j$. This subgroup
is isomorphic to the direct product of all $S_{\gamma_{ij}}$ and hence
it has $\gamma !$ elements.
\end{proof}

\begin{lemma}
  We have $$\sum_{\gamma} \frac{\lambda!\mu!}{\gamma!} = n!,$$ 
  where the sums is over all matrices $\gamma$ of size $r\times s$ with non-negative
  entries satisfying condition \eqref{rc-constraints}.
\end{lemma}
\begin{proof}
Let us compute
\ben
(x_1+\cdots + x_s)^n=
(x_1+\cdots +x_s)^{\lambda_1}\cdots (x_1+\cdots + x_s)^{\lambda_r}
\een
in two different ways. First, using that
\ben
(x_1+\cdots +x_s)^{\lambda_i} = \sum_{\gamma_{i1}+\cdots +\gamma_{is}=
  \lambda_i}
\frac{\lambda_i !}{\gamma_{i1}!\cdots \gamma_{is}!} \,
x_1^{\gamma_{i1}}\cdots x_s^{\gamma_{is}}
\een
we get that the above expression coincides with
\ben
\sum_{\mu_1+\cdots +\mu_s=n} \sum_{\gamma} \frac{\lambda!}{\gamma!}
x_1^{\mu_1}\cdots x_s^{\mu_s}.
\een
On the other hand, the coefficient in front of
$x^\mu:=x_1^{\mu_1}\cdots x_s^{\mu_s}$ is $n!/\mu!$. Comparing the two
formulas, we get the identity that we wanted to prove.
\end{proof}

We proved the following proposition.
\begin{proposition}\label{prop:dc}
  Suppose that $\lambda$ and $\mu$ are partitions of $n$ of sizes
  respectively $r$ and $s$. 
The double cosets in $S_\mu\backslash S_n /S_\lambda$ are parametrized
by matrices $\gamma=(\gamma_{ij})$ of size $r\times s$ whose entries
are non-negative integers satisfying condition
\eqref{rc-constraints}. Moreover, let us associate with each $\gamma$
a permutation $\sigma$ as explained above, then  the resulting set of
permutations gives a complete set of representatives for the double
cosets in $S_\mu\backslash S_n /S_\lambda$.  
\end{proposition}
Note that in Proposition \ref{prop:dc} the condition that $\lambda$ and $\mu$ are partitions can be relaxed, that is, we do not have to require that $\lambda=(\lambda_1,\dots,\lambda_r)$ and $\mu=(\mu_1,\dots,\mu_s)$ are decreasing sequences. 

\subsection{Induced vector bundles}\label{sec:ind}
We will assume that the reader is familiar with the notion of a
$G$-space and $G$-vector bundle. For some background, we refer to
\cite{Segal1968}. 
Suppose that $G$ is a finite group, $H$ a subgroup of $G$, and $X$ is
a $G$-space. For every $H$-bundle on $X$ we define a $G$-bundle
$\op{Ind}_H^G(E)$ on $X$ whose points are the set theoretic maps
$f:G\to E$ satisfying the following two conditions
\begin{enumerate}
\item[(i)]
  The composition $\pi\circ f: G\to X$ is $G$-equivariant, that is,
  $\pi(f(g))=g\pi(f(1))$ for all $g\in G$ where $\pi:E\to X$ is the
  projection map.
\item[(ii)]
  The map $f$ is $H$-equivariant, that is, $f(hg)=hf(g)$ for all $h\in
  H$ and $g\in G$. 
\end{enumerate}
Note that when $X$ is a pont we have: $E$ is a representation of $H$
and $\op{Ind}_H^G(E)$ is the induced representation. The structure of
a $G$-bundle on $\op{Ind}_H^G(E)$ is defined as follows. First, the
$G$-action on $\op{Ind}_H^G(E)$ is defined by $(gf)(g'):=f(g'g)$.  The
structure projection $p:\op{Ind}_H^G(E)\to X$ is defined by $f\mapsto
\pi(f(1))$.  It is strightforward to check that $p$ is a
$G$-equivariant map. Suppose that $f:G\to E$ is in the fiber
$p^{-1}(x)$. Then $\pi(f(g))= g \pi(f(1))=gp(f)=gx$, that is,  
$f(g)\in E_{gx}$ where $E_y$ denotes the fiber of $E$ at $y$. Using the
linear strucure on the fibers of $E$ we get that $p^{-1}(x)$ is
naturally a linear vector space: $(f_1+f_2)(g):=f_1(g)+f_2(g)$ and
$(cf)(g):=cf(g)$. Let $g_1,g_2,\dots,g_s$ be a complete set of
representatives of the right coset classes in $H\backslash G$. Since
$f(hg)=hf(g)$ we see that the map $f$ is uniquely determined by its
values  $f(g_i)\in E_{g_ix}$. This proves that $p^{-1}(x)$ is a finite
dimensional vector space of dimension $\op{rk}(E) |G:H|$ where $|G:H|$ is
the index of $H$ in $G$. Moreover, there is a set-theoretic
isomorphism 
\beq\label{vb-str}
\phi: \bigoplus_{j=1}^s g_j^*E \to \op{Ind}_H^G(E), 
\quad
\phi (v_1,\dots, v_s) (g)= gg_i^{-1} \cdot \widetilde{g}_i(v_i),
\eeq
where $i=i(g)$ is the unique index, such that, $g\in H g_i$
(note that $gg_i^{-1}\in H$ acts on $E$) and
$\widetilde{g}_i$ is the map defined by the following Cartesian
square:
\ben
\xymatrix{
  g_i^*E \ar[r]^-{\widetilde{g}_i}\ar[d]_{g_i^*\pi} & E\ar[d]^\pi \\
  X\ar[r]^{g_i} &  X
  }
\een
where slightly abusing the notation we denote by $g_i:X\to X$ the map
defined by $x\mapsto g_i x$. Using the map \eqref{vb-str}, we equip
$\op{Ind}_H^G(E)$ 
with the structure of a vector bundle over $X$. It is easy to check
that the $G$-action on $\op{Ind}_H^G(E)$ defined above is continuous
and therefore $\op{Ind}_H^G(E)$  is a topological $G$-vector bundle on
$X$.

We also have a restriction functor. Namely, if $X$ is a $G$-space,
$H\leq G$ is a subgroup, and $E$ is a $G$-vector bundle, then $E$ is
also an $H$-vector bundle which will be denoted by
$\op{Res}^G_H(E)$. The Frobenius reciprocity rules take the following
form.
\begin{proposition}[Frobenius reciprocity]\label{prop:Frob}
Suppose that $X$ is a $G$-space and that $H\leq G$ is a finite
subgroup. Let $E$ be a $G$-vector bundle and $F$ be an $H$-vector
bundle. 

a) We have an isomorphism
\ben
\op{Hom}_G(E,\op{Ind}_H^G(F))\cong \op{Hom}_H(\op{Res}^G_H(E),F)
\een
uniquely determined  by the following relation:
\ben
J(e)(g)=j(ge),\quad e\in E,\quad g\in G,
\een
where $J\in \op{Hom}_G(E,\op{Ind}_H^G(F))$ and $j\in
\op{Hom}_H(\op{Res}^G_H(E),F)$ are maps that correspond to each other
via the above isomorphism.

b) We have an isomorphism
\ben
\op{Hom}_G(\op{Ind}_H^G(F), E)\cong \op{Hom}_H(F,\op{Res}^G_H(E))
\een
uniquely determined  by the following relation:
\ben
J(f)=\sum_{\gamma\in G/H}
\gamma\, j\circ f(\gamma^{-1}),\quad
f\in \op{Ind}_H^G(F), 
\een
where $J\in \op{Hom}_G(\op{Ind}_H^G(F),E)$ and $j\in
\op{Hom}_H(F,\op{Res}^G_H(E))$ are maps that correspond to each other
via the above isomorphism.
\end{proposition}
If $X$ is a point, Proposition \ref{prop:Frob} is well known in the
representation theory of finite groups (see \cite{Bump2014}, Section
34). The proofs in general remains 
the same. The following proposition is Lemma 7 in \cite{Wang2000}.  
\begin{proposition}\label{prop:IR-RI}
  Suppose that $X$ is a $G$-space, $H_1$ and $H_2$ are finite
  subgroups of $G$, and $W$ is a $H_1$-vector bundle. For each 
  $s\in  G$, let $K_s:=\{(h_1,h_2)\in H_1\times
  H_2\ |\ h_1s=sh_2 \}$. Using the projection on the $i$-th factor
  $H_1\times H_2\to H_i$, 
  we view $K_s$ as a subgroup of $H_i$ for $i=1,2$. We have the following
  relation in $K_{H_2}(X)$: 
  \ben
  \op{Res}_{H_2}^G \, \op{Ind}_{H_1}^G (W)=
  \sum_{s\in H_1\backslash G/H_2}
  \op{Ind}_{K_s}^{H_2} \, s^*\, \op{Res}^{H_1}_{K_s} (W),
  \een
  where the sum is over a set of elements $s\in G$ which represent the
  double coset classes in $H_1\backslash G/H_2$.
\end{proposition}
In the case when $X$ is a point, the statement is well known (see
\cite{Serre2012}, Proposition 22). Serre's argument works in the
general case too.   

\subsection{Anihilation operators}\label{sec:ca_operators}
In order to define the
anihilation operations, let us first recall the 
notion of a trace of a vector bundle. Suppose that $E$ is a complex vector
bundle on some compact topological space $Y$ and that $g:E\to E$ is a
finite order 
automorphism acting trivially on the base $Y$. To be more precise, we
have $g(y)=y$ for all $y\in Y$ and the induced map $g_y: E_y\to E_y$ is a finite order
linear map. In particular, the maps $g_y$ are diagonalizable and their
eigenvalues are independent of $y$. The {\em trace} of $g$ on $E$ is a
virtual vector bundle on $Y$ defined by 
\ben
\op{Tr}_g(E)=\sum_{\lambda\in \CC} \lambda E_\lambda,\quad
E_\lambda:= \op{Ker}(\lambda-g: E\to E),
\een
where only finitely many terms in the above sum are non-zero
because $\lambda$ must be an eigenvalue of $g$ and $E_\lambda$ is a
vector bundle because the eigenvalues of $g_y$ and their
multiplicities are independent of $y$.  The following lemma is well
known (see \cite{Atiyah1966}, Section 2).
\begin{lemma}\label{le:Adams}
  Let $E$ be a complex vector bundle on $X$, then
 $\psi^n(E):=\op{Tr}_{(1,2,\dots,n)}(E^{\otimes n})$ is the Adam's
 operation.  
\end{lemma}
\begin{proof}
The operation is compatible with pullback, that is,
$\psi^n(f^*E)=f^*\psi^n(E)$ for all continuous maps $f:X\to
Y$. Recalling the splitting principle definition of the Adam's
operations, we get that it is sufficient to 
prove that if $E=L_1+\cdots + L_r$ is a direct sum of
line bundles, then $\psi^n(E)=L_1^n+\cdots + L_r^n$. This is obvious
because in the tensor product
\ben
E^{\otimes n}= \oplus_{1\leq i_1, \cdots, i_n\leq r}
L_{i_1}\otimes \cdots \otimes L_{i_n}
\een
all terms for which the sequence $(i_1,\dots,i_{n-1},i_n)$ is not invariant
under the cyclic permutation do not contribute to the trace. On the
other hand, the sequence equals its cyclic permutation
$(i_2,\dots,i_n,i_1)$ iff $i_1=i_2=\dots=i_n$. 
\end{proof}

Using the Chern character map, we
get that the Adam's operation $\psi^m:K(X)\to K(X)$ is an
isomorphism where recall that we work with the K-ring with complex
coefficients. Therefore, for every $F\in K(X)$, there exists an
$E\in K(X)$, such that, $\psi^m(E)=F$.  
We will need also the formula for the trace of an induced vector
bundle. 
\begin{proposition}\label{prop:induced_trace}
Suppose that $G$ is a finite group, $H\leq G$ is a subgroup, and $Y$ is a
trivial $G$-space.  For every $H$-vector bundle on $Y$ we have the
following formula:
\ben
\op{Tr}_g (\op{Ind}_H^G(E)) = \sum_{i: g_i g g_i^{-1}\in H}
\op{Tr}_{g_i g g_i^{-1}} (E),
\een
where $g_1,\dots,g_k\in G$ is a complete set of representatives for the
right coset classes $H\backslash G$.  
\end{proposition}
\begin{proof}
Since $G$ acts trivially on $Y$, we have $\op{Ind}_H^G(E)\cong
\oplus_{i=1}^k E$ where the isomorphism is given by $f\mapsto
(f(g_1),\dots,f(g_k))$. Let $\sigma\in S_k$ be the permutation defined
by  $Hg_ig=Hg_{\sigma(i)}$ and let $h_i\in H$ be defined by
$g_ig=h_ig_{\sigma(i)}$. If $I=\{i_1,\dots,i_r\}$ is an orbit of
$\sigma$, then the subspace $\oplus_{i\in I} E$ is
$g$-invariant. Therefore,
\ben
\op{Tr}_g(\op{Ind}_H^G(E))=\sum_{I} \op{Tr}_g\Big( \bigoplus_{i\in I} E\Big),
\een
where the sum is over all orbits $I$ of $\sigma$ in
$\{1,2,\dots,k\}$. Note that the action of $g$ on $\oplus_{i\in I} E$
is represented by a $I\times I$ matrix $H_I$ with entries in $\op{End}(E)$
of the following type
\ben
H_I:=
\begin{bmatrix}
  0 &  H_{i_1} & 0 & \cdots & 0 \\
  0 & 0 & H_{i_2} & \cdots & 0 \\
  \vdots &\vdots & \ddots & \ddots&\vdots \\
   0 & 0 & \cdots & 0 & H_{i_{r-1}}\\
  H_{i_r} & 0 & \cdots &   0  & 0
\end{bmatrix},
\een
where $\{i_1,\dots,i_r\}:= I$ and $H_i\in \op{End}(E)$ is the operator
representing the action of $h_i=g_i g g_{\sigma(i)}^{-1}$.  Note that
\ben
\op{Tr}_g\Big( \oplus_{i\in I} E\Big) = \op{Tr}_{h_I} (E),
\een
where $h_I:=\op{Tr}(H_I)\in \op{End}(E)$. Since $h_I=0$ for $r>1$, we
get that only the one-point orbits $I=\{i\}$ of $\sigma$ contribute to
the trace of $g$. It remains only to notice that the one point orbits
$I=\{i\}$, that is, the fixed points of $\sigma$, correspond precisely
to those $i$ for which $g_i g g_{i}^{-1}\in H$ and that $h_I= g_i g
g_{i}^{-1}$. 
\end{proof}

Suppose that $W\in K(X)$ is a virtual vector bundle on $X$. Let us
define the contraction operation 
\ben
\xymatrix{
\iota_{-m}(W): K_{S_m}(X^m)\ar[r] &  \CC,} 
\quad 
E\mapsto \op{tr}_{(1,2,\dots,m)}\pi_*(E\otimes W^{\boxtimes m}),
\een
where $\pi:X^m\to \op{pt}$ is the contraction map and $\pi_*$ is the K-theoretic pushforward, that is, 
\ben
\iota_{-m}(W)(E):=
\sum_{i=0}^\infty (-1)^i
\op{tr}_{(1,2,\dots,m)} H^i(X^m, E\otimes W^{\boxtimes m}). 
\een
In the next lemma we will use the Lefschetz trace formula. We refer to
Givental's work (see \cite{Giv2017}, Section 2) for a proof
of the Lefscetz trace formula based on the Kawasaki's Riemann--Roch
formula (see \cite{Ka1979}).  
\begin{lemma}\label{le:contraction}
  Suppose that $E$ is a $S_m$-bundle on $X^m$. Then
  \ben
  \iota_{-m}(W) (E)=\chi(F\otimes W),
  \een
  where $F\in K(X)$ is such that
  $\psi^m(F)=\op{Tr}_{(1,2,\dots,m)} (E|_{X})$ where $X\subset X^m$ is
  the diagonal of $X^m$, that is, the submnaifold
  $\{x\in X^m\ |\ x_1=\cdots =x_m\}$. 
\end{lemma}
\proof
Put $c=(1,2,\dots,m)\in S_m$. According to the Lefschetz trace
formula, $\iota_{-m}(W)(E)$ coinicdes with the Euler characteristics
of the following vector bundle
\beq\label{c-trace}
\frac{\op{Tr}_c(E|_X\otimes W^{\otimes m}) }{
  \op{str}_c( \wedge^\bullet (N^\vee_{X^m|X}))},
\eeq
where $N^\vee_{X^m|X}$ is the conormal bundle to $X$ in $X^m$ and
  \ben
  \op{str}_c( \wedge^\bullet (N^\vee_{X^m|X}))=
  \sum_{i=0}^\infty (-1)^i \op{Tr}_c  \wedge^i (N^\vee_{X^m|X}).
  \een
The numerator in \eqref{c-trace} is precisely $\psi^m(F\otimes
W)$. The denominator can be expressed in terms of the K-theoretic
Chern roots of $TX$. Indeed, we have
\ben
\xymatrix @R=.5pc{
  0\ar[r] & TX \ar[r] & TX\otimes \CC^m \ar[r] & N_{X^m|X}\ar[r] & 0.
  \\
  & v \ar@{|->}[r] & v\otimes (1,1,\dots,1) & &
}
\een
We get $N_{X^m|X}=TX\otimes \CC^{m-1}$, where $\CC^{m-1}$ is
identified with the hyperplane $x_1+\cdots +x_m=0$ in $\CC^m$ and the
action of $c$ on $N_{X^m|X}$ corresponds to permuting cyclically the
coordinates $(x_1,\dots,x_m)\in \CC^{m-1}$. If $L_1,\dots, L_D$ are
the K-theoretic Chern roots of $TX$, then we get
\ben
N^\vee_{X^m|X} =\bigoplus_{i=1}^D \bigoplus_{j=1}^{m-1} L_i^{-1}
\otimes v_j,
\een
where $v_j\in \CC^{m-1}$ is the eigenvector of $c$ with eigenvalue
$\eta^j$ where $\eta=e^{2\pi\ii/m}$. The denominator in
\eqref{c-trace} takes the form
\ben
\op{str}_c( \wedge^\bullet (N^\vee_{X^m|X})) =
\prod_{i=1}^D \prod_{j=1}^{m-1} (1-L^{-1}_i \eta^j).
\een
Recalling the Hierzerbruch--Riemann--Roch formula, we get that the
Euler characterisics of \eqref{c-trace} is
\ben
\int_X \prod_{i=1}^D \frac{x_i}{1-e^{-x_i}}\,
\frac{\op{ch}(\psi^m(F\otimes W))}{
  \prod_{i=1}^D \prod_{j=1}^{m-1} (1-e^{-x_i} \eta^j)}=
\int_X \prod_{i=1}^D \frac{x_i}{1-e^{-mx_i}}\,
\op{ch}(\psi^m(F\otimes W)).
\een
Let $\op{deg}:H^*(X)\to H^*(X)$ be the complex degree operator, that
is, $\op{deg}(\phi)=i \phi$ for $\phi\in H^{2i}(X)$. Note that
$\op{ch}\circ\psi^m = m^{\op{deg}} \circ \op{ch}$ and that
$\int_X\alpha=
m^{D} \int_X m^{-\op{deg}} \alpha$. Therefore, the above formula
coincides with
\ben
m^{D} 
\int_X \prod_{i=1}^D \frac{m^{-1}x_i}{1-e^{-x_i}}\,
\op{ch}(F\otimes W) = \chi(F\otimes W).
\qed
\een

\begin{corollary}
We have $\iota_{-m}(W)(\nu_{m,\alpha})=m \,\chi(W\otimes \Phi^\alpha)$. 
\end{corollary}
\proof
We have
\ben
\op{Tr}_c((\Phi^\alpha)^{\otimes m} \otimes p_m)=
\op{Tr}_c((\Phi^\alpha)^{\otimes m} \otimes \op{Tr}_c( p_m) =
m\psi^m(\Phi^\alpha)=\psi^m(m\Phi^\alpha).
\een
It remains only to recall Lemma \ref{le:contraction}.
\qed

The contraction operation extends to the entire Fock space. Namely,
let us define
\ben
\xymatrix{
\iota_{-m}(W): K_{S_n}(X^n)\ar[r] &  K_{S_{n-m}}(X^{n-m})}
\een
by composing the restriction and the anihilation operations as follows:
\ben
\xymatrixcolsep{5pc}
\xymatrix{
K_{S_n}(X^n) \ar[r]^-{
\op{Res}_{S_m\times S_{n-m}}^{S_n}} & 
K_{S_m}(X^m)\otimes K_{S_{n-m}}(X^{n-m}) \ar[r]^-{
  \iota_{-m}(W)\otimes \op{id} } &
K_{S_{n-m}}(X^{n-m}).}
\een
\begin{theorem}\label{thm:Heisenberg}
  Under the isomorphism $\F(X)\cong \CC[\nu_{m,\alpha}(1\leq \alpha\leq
  m, n\geq 0)]$, we have
  $\iota_{-m}(\Phi_\alpha)=m\, \partial/\partial \nu_{m,\alpha}$. 
\end{theorem}
\proof
Let us compute the composition
\ben
\xymatrix{
  \iota_{-m}(W) \circ \nu_{l,\alpha} :
  K_{S_{n'}}(X^{n'})\ar[r] &
  K_{S_{n''}}(X^{n''})
}
\een
where $k:=n'+l=m+n''$. By definition
\ben
\iota_{-m}(W) (\nu_{l,\alpha} E)=
\iota_{-m}(W)\otimes \op{id}\, \circ \, 
\op{Res}^{S_k}_{S_m\times S_{n''}} \circ
\op{Ind}^{S_k}_{S_l\times S_{n'}} (
V^{\boxtimes l}\otimes p_l \boxtimes E),
\een
where we put $V=\Phi^\alpha$ for the ease of notation.  
We would like to use Proposition \ref{prop:IR-RI} with $G=S_k$,
$H_1:=S_l\times S_{n'}$ and $H_2=S_m\times S_{n''}$. The double cosets 
in $S_l\times S_{n'} \backslash G/S_m\times S_{n''}$, according to
Proposition \ref{prop:dc}, are parametrized by the set of $2\times 2$
matrices
\ben
\begin{bmatrix}
  a & b \\
  c & d
\end{bmatrix},
\quad
\begin{matrix}
  a+b=m & a+c =l,\\
  c+d=n'' & b+d = n',
\end{matrix}
\een
where $a,b,c,d\in \ZZ_{\geq 0}$. Moreover, if $s\in S_k$ is the
representative of the double coset $S_l\times S_{n'} \backslash
G/S_m\times S_{n''}$ corresponding to a matrix of the above form (see
Proposition \ref{prop:dc}), then
\ben
K_{1,s}=H_1\cap sH_2 s^{-1} = S_a\times S_c \times S_b\times S_d
\een
and
\ben
K_{2,s}=H_2\cap s^{-1}H_1 s = S_a\times S_b \times S_c\times S_d.
\een
In other words, $s\in S_k$ is the permutation that fixes the first $a$
and the last $d$ numbers in $1,2,\dots,k$, that is,
\ben
s(i)=i\ (1\leq i\leq a), \quad
s(a+b+c+j)=a+b+c+j\ (1\leq j\leq d),
\een
and that switches the two blocks $a+1,\dots, a+b$ and $a+b+1,\dots,
a+b+c$, that is, 
\ben
s(a+i)=a+c+i \ (1\leq i\leq b),\quad
s(a+b+j)=a+j\ (1\leq j\leq c).
\een
According to Proposition \ref{prop:IR-RI} we have 
\beq\label{RI}
\op{Res}^{S_k}_{S_m\times S_{n''}} \circ
\op{Ind}^{S_k}_{S_l\times S_{n'}} (
V^{\boxtimes l}\otimes p_l \boxtimes E)=
\sum \op{Ind}^{S_m\times S_{n''}}_{
  S_a\times S_b\times S_c\times S_d} \circ
s^*\circ
\op{Res}^{S_l\times S_{n'}}_{
  S_a\times S_c\times S_b\times S_d}
(
V^{\boxtimes l}\otimes p_l \boxtimes E),
\eeq
where the sum is over all $2\times 2$ matrices of the above type
representing the double cosets in $H_1\backslash G/H_2$.  
Let us examine the restriction
\beq\label{restriction}
\op{Res}^{S_l\times S_{n'}}_{S_a\times S_c\times S_b\times S_d}(
V^{\boxtimes l}\otimes p_l \boxtimes E) .
\eeq
Since $\op{Res}^{S_l}_{S_a\times S_c}(p_l)$ is non-zero only if
$S_l=S_a\times S_c$, we get that the above restriction is non-zero
only in two cases: $a=0$, $c=l$ or $a=l$, $c=0$. Let us analyze these
two cases separately.

First, suppose that $a=0$ and $c=l$. Then $b=m$ and $d=n'-m=n''-l$.
We get
\ben
\op{Res}^{S_l\times S_{n'}}_{S_a\times S_c\times S_b\times S_d}(
V^{\boxtimes l}\otimes p_l \boxtimes E) ) =
V^{\boxtimes l}\otimes p_l \boxtimes
\op{Res}^{S_{n'}}_{S_m\times S_{n'-m}} (E).
\een
The contribution to \eqref{RI} becomes
\ben
\op{Ind}^{S_m\times S_{n''}}_{
  S_m\times S_l\times S_{n''-l}} \circ
s^*(
V^{\boxtimes l}\otimes p_l \boxtimes
\op{Res}^{S_{n'}}_{S_m\times S_{n'-m}} (E)).
\een
Applying to the above expression the contraction operation
$\iota_{-m}(W)$ we get $\nu_{l,\alpha} \, \iota_{-m}(W) (E)$.

The second case is $a=l$ and $c=0$. Then $b=m-l=n'-n''$ and $d=n''$. The
restriction \eqref{restriction} becomes
\ben
\op{Res}^{S_l\times S_{n'}}_{S_a\times S_c\times S_b\times S_d}(
V^{\boxtimes l}\otimes p_l \boxtimes E) ) =
V^{\boxtimes l}\otimes p_l \boxtimes
\op{Res}^{S_{n'}}_{S_{n'-n''}\times S_{n''}} (E).
\een
The contribution to \eqref{RI} becomes (note that now $s=\op{id}$)
\ben
\op{Ind}^{S_m\times S_{n''}}_{
  S_l\times S_{m-l}\times S_{n''}} \Big(
V^{\boxtimes l}\otimes p_l \boxtimes
\op{Res}^{S_{n'}}_{S_{m-l}\times S_{n''}} (E)\Big).
\een
Note that when we apply the contraction operation
$\iota_{-m}(W)$ we will get a sum of terms that have a factor of the form
$\op{Tr}_{(1,2,\dots,m)}\op{Ind}^{S_m}_{S_l\times S_{m-l}}(V^{\otimes
  l}\otimes p_l\otimes F)$ where $F$ is a representation of $S_{m-l}$,
that is, a $S_{m-l}$-vector bundle on $X$ for which $S_{m-l}$ acts
trivially on $X$.
This trace is $0$ for $m\neq l$. Indeed, recalling the formula for the
trace of an induced vector bundle (see Proposition
\ref{prop:induced_trace}), we get that the trace must be 
a sum over traces of elements in $S_l\times S_{m-l}$ of the form 
$x(1,2,\dots,m) x^{-1}$. The latter is a cycle of length $m$. However, the
group $S_l\times S_{m-l}$ contains such elements only if
$m=l$. Therefore, the contribution in this case is non-zero only if
$m=l$, that is,
\ben
(\iota_{-m}(W)\otimes \op{id}) \,
(V^{\boxtimes m}\otimes p_m \boxtimes E) =
m\, \chi(W\otimes V) \, E.
\een
Recall that $V=\Phi^\alpha$ and let us specalize $W=\Phi_\beta$. Then
the above computations imply the following commutation relations:
\ben
[\iota_{-m}(\Phi_\beta), \nu_{l,\alpha}] =m\delta_{m,l} \delta_{\alpha,\beta}.
\een
The statement in the theorem follows.
\qed


\section{Permutation-equivariant K-theoretic Gromov--Witten theory}

Suppose now that $X$ is a smooth projective variety. We would like to
recall the background on permutation-equivariant KGW theory. Let
\ben
\F(X)\cong
\CC[\nu_{m,\alpha}(1\leq \alpha \leq N, m\geq 1)]
\een
be the $K$-theoretic Fock space of $X$.

\subsection{Translation symmetry}
The correlators \eqref{nu-cor} have the following symmetry:
\beq\label{KGW-ti}
\frac{\partial}{\partial \nu_{1,\alpha}} \langle
\mathbf{t}(L_1),\dots,\mathbf{t}(L_n)\rangle_{0,n}(\nu)=
\langle \mathbf{t}(L_1),\dots,\mathbf{t}(L_n),\Phi_\alpha \rangle_{0,n+1}(\nu).
\eeq
The derivative on the LHS is equivalent to the operation
$\iota_{-1}(\Phi_\alpha)$. The correlator has the form
$\op{ev}^{(k)}_*(A\otimes \op{ev}_{(n)}^*B)$. Its derivative becomes
\ben
\pi_*\left(
  \Phi_\alpha \boxtimes 1^{\boxtimes k-1} \otimes
  \op{ev}^{(k)}_*(A\otimes \op{ev}_{(n)}^*B)\right) =
\pi_*\left(
\op{ev}^{(k)}_*(A\otimes \op{ev}_{(n)}^*B\otimes
\op{ev}_{n+1}^*\Phi_\alpha )\right)
\een
where $\pi:X^k=X\times X^{k-1}\to X^{k-1}$ is the projection map and
we used that
\ben
(\op{ev}^{(k)})^*(\Phi_\alpha \boxtimes 1^{\boxtimes k-1} )=
\op{ev}_{n+1}^*\Phi_\alpha.
\een
Note that $\pi\circ \op{ev}^{(k)}=\op{ev}^{(k-1)}$ and
$\op{ev}_{(n)}^*B\otimes
\op{ev}_{n+1}^*\Phi_\alpha=\op{ev}_{(n+1)}^*(B\boxtimes
\Phi_\alpha)$.

Using formula \eqref{KGW-ti}, we get that no information is lost if we
set $\nu_{1,\alpha}=0$ for all $\alpha$. Indeed, if we do this, then
in order to recover the dependence on $\nu_{1,\alpha}$ we simply have
to translate $t_{0,\alpha}\mapsto t_{0,\alpha}+\nu_{1,\alpha}$.  

\subsection{Givental's permutation-equivariant correlators}
\label{sec:pec}
We would like to compare our definition to the permutation-equivariant
KGW invariants of Givental. The definition evolved on 3 stages and the
most general one appears in \cite{Giv2017}, Section 1. Givental's
correlators depend on the choice of a $\lambda$-algebra $\Lambda$. The
inputs of the correlators are parametrized by a sequence of formal
variables $\mathbf{t}=(\mathbf{t}_1,\mathbf{t}_2,\dots)$, 
$\mathbf{t}_r=\{t_{r,k,\alpha}\}$ where the index set is defined by $r\geq 1, k\in
\ZZ,$ and $1\leq \alpha\leq N$ and $t_{r,k,\alpha}\in \Lambda$. Let
$\mathbf{l}=(l_1,l_2,\dots)$ be a sequence such that $l_i\geq 0$ and
$l_i>0$ for only finitely many $i$. Givental introduced correlators of the following form:
\beq\label{corr-giv}
\langle 
\underbrace{\mathbf{t}_1,\dots,\mathbf{t}_1}_{l_1};
\underbrace{\mathbf{t}_2,\dots,\mathbf{t}_2}_{l_2};\dots ;
\underbrace{\mathbf{t}_r,\dots,\mathbf{t}_r}_{l_r};\dots
\rangle_{g,\mathbf{l},d},
\eeq
where $\mathbf{t}_i$ $(i=1,2,\dots)$ is repeated $l_i$ times, that is,
the correlator has $l_1+l_2+\cdots$ inputs. Put $n=\sum_{r=1}^\infty r
l_r$ and let us fix a permutation $h\in S_n$ whose cycle type is
determined by $\mathbf{l}$, that is, the action of $h$ on $\{1,2,\dots,n\}$ has exactly 
$l_r$ orbits of length $r$ for all $r=1,2,\dots$. Here, for a given
element $h\in S_n$ and $i\in \{1,2,\dots,n\}$ let $r>0$ be the minimal positive
integer, such that, $h^r(i)=i$. We refer to the sequence
$(i,h(i),h^2(i),\dots,h^{r-1}(i))$ as the  orbit of $h$ through
$i$ and to the number $r$ as the length of the orbit. The set of all
orbits of $h$ will be denoted by $O(h)$. The correlator is by
definition the following sum of super-traces  
\beq\label{str}
\sum_{\alpha=(\alpha_I),\beta=(\beta_I)}
\op{str}_h \ 
H^*\Big(\overline{\M}_{g,n}(X,d), \O_{g,n,d}\otimes 
\otimes_{I\in O(h)} \otimes_{i\in I} 
(L_i^{\beta_I} \op{ev}_i^*\Phi_{\alpha_I}) \Big)\ 
\prod_{I\in O(h)} \frac{1}{r_I}
\Psi^{r_I}(t_{r_I,\beta_I,\alpha_I}),
\eeq
where the sum is
over all sequences $\alpha=(\alpha_I)$ and $\beta=(\beta_I)$
consisting of integers $1\leq \alpha_I\leq N$ and $\beta_I\in \ZZ$ -- one
for each orbit $I\in O(h)$, $r_I$ is the number of elements in the
orbit $I\in O(h)$, and $\Psi^r$ is the $r$-th Adam's operation in
$\Lambda$. Note that the definition of the correlator
\eqref{corr-giv} is independent of the choice of $h$, that is, it
depends only on the cycle type $\mathbf{l}$ of $h$.
Finally, the genus-$g$ potential of permutation-equivariant theory is
defined by 
\ben
\F^{(g)}(\mathbf{t}) =\sum_{d} Q^d \sum_{\mathbf{l}}
\frac{1}{\prod_r l_r!}\,
\langle 
\underbrace{\mathbf{t}_1,\dots,\mathbf{t}_1}_{l_1};
\underbrace{\mathbf{t}_2,\dots,\mathbf{t}_2}_{l_2};\dots ;
\underbrace{\mathbf{t}_r,\dots,\mathbf{t}_r}_{l_r};\dots
\rangle_{g,\mathbf{l},d}.
\een
In order to compare to our definition, let us put $t_{r,k,\alpha}=0$
for all $r>1$ and $k\neq 0$, that is, let us consider only invariants that do
not involve descendants at the permutable marked points. Also, let us
restrict \eqref{nu-cor} to $\nu_{1,\alpha}=0$. The super-trace in \eqref{str} becomes 
\beq\label{trace_h}
\op{tr}_h\, 
\op{\pi}_* \left(
  \boxtimes_{I\in O_{>1}(h)} 
  \Phi_{\alpha_I}^{\boxtimes r_I} 
  \op{ev}^{(k)}_*\Big(
  \prod_{i\in O_1(h)} (L_i^{\beta_i}\op{ev}^*(\Phi_{\alpha_i})
  \Big)
  \right), 
\eeq
where $O_{>1}(h)$ (resp. $O_1(h)$) denotes the set of orbits of $h$ of
length $>1$ (resp. $1$), $k=2l_2+3l_3+\cdots$, and $\pi:X^k\to
\op{pt}$ is the contraction map. Without loss of
generality, we may assume that $O_1(h)=\{(1),(2),\dots,(l_1)\}$.
Recalling Theorem \ref{thm:Heisenberg} we get that the trace
\eqref{trace_h} coincides with
\ben
\left.
\prod_{I\in O_{>1}(h)}  (r_I \partial_{\nu_{r_I,\alpha_I}}) \,
\langle \Phi_{\alpha_1} L^{\beta_1},\dots,\Phi_{\alpha_{l_1}}
L^{\beta_{l_1}}\rangle_{g,l_1,d} (\nu) \right|_{\nu=0}= 
C_{\mathbf{m}}(\alpha_1,\beta_1,\dots,\alpha_{l_1},\beta_{l_1})\,
\prod_{r>1}\prod_{\alpha=1}^N r^{m_{r,\alpha}}\, m_{r,\alpha}! 
\een
where $\mathbf{m}=(m_{r,\alpha})$, $m_{r,\alpha}:=\op{card}\{I\in O_r(h)\ |\ \alpha_I=\alpha\}$, and $C_{\mathbf{m}}(\alpha_1,\beta_1,\dots,\alpha_{l_1},\beta_{l_1})$ is the coefficient in front of the monomial $\prod_{r>1,\alpha} \nu_{r,\alpha}^{m_{r,\alpha}}$ in the correlator 
$
\langle \Phi_{\alpha_1} L^{\beta_1},\dots,
\Phi_{\alpha_{l_1}} L^{\beta_{l_1}}\rangle_{g,l_1,d} (\nu) $. Note
that $\sum_{\alpha=1}^N m_{r,\alpha}=l_r$ and that the number of
sequences $\alpha=(\alpha_I)$ and $\beta=(\beta_I)$ in the sum in
\eqref{str}, such that,  
\ben
\prod_{I\in O(h)} 
\Psi^{r_I}(t_{r,\beta_I,\alpha_I}) = 
t_{1,\beta_1,\alpha_1}\cdots t_{1,\beta_{l_1},\alpha_{l_1}}
\prod_{r>1}\prod_{\alpha=1}^N
\Psi^r(t_{r,0,\alpha})^{m_{r,\alpha}}
\een
is precisely $\prod_{r>1}\tfrac{l_r!}{m_{r,1}!\cdots m_{r,N}!}$ because the values of $\beta_I$ and $\alpha_I$ are fixed for the one-point orbits $I\in O_1(h)$ and for the more than one-point-orbits, we have $\beta_I=0$ and we have to choose $m_{r,\alpha}$ orbits $I$ in $O_r(h)$ for which $\alpha_I$ will be assigned value $\alpha$. Now it is clear that the coefficient in front of $t_{\beta_1,\alpha_1}\cdots t_{\beta_{l_1},\alpha_{l_1}}
\prod_{r>1}\prod_{\alpha=1}^N
\Psi^r(t_{r,0,\alpha})^{m_{r,\alpha}}$ in \eqref{str} is 
$C_{\mathbf{m}}(\alpha_1,\beta_1,\dots,\alpha_{l_1},\beta_{l_1}) \prod_{r>1} l_r! $. Therefore, under the substitutions
\beq\label{substitution}
\nu_{1,\alpha}=0 ,\quad
\nu_{r,\alpha}:=\Psi^r(t_{r,0,\alpha})\ (r>1), \quad
t_{k,\alpha}:=t_{1,k,\alpha} \ (k\in \ZZ),\quad 
t_{r,k,\alpha}:=0 \ (r>1, k\neq 0),
\eeq
where $1\leq \alpha\leq N$, 
we get that the genus-$g$ potential $\F^{(g)}(\mathbf{t})$ coincides with 
\ben
\sum_{n=0}^\infty 
\frac{1}{n!}
\langle \mathbf{t}(L_1),\dots,\mathbf{t}(L_n)\rangle_{g,n}(\nu).
\een
Therefore, definition \eqref{nu-cor} contains all the information for
permutation-equivariant KGW invariants that do not involve descendants
at the permutable marked points.

\subsection{Descendants at the permutable entries}
\label{sec:perm-desc}

Let us sketch the necessary extension of the Fock
space which will alow us to keep track of descendants. The reader not
interested in the construction can skip this section as the results
here would not be used in the rest of the paper.  The main point 
is that we have to introduce the K-theoretic Fock space of
$X\times \ZZ$ where $\ZZ$ is equipped with the discrete topology. Note
that this is a topological space with infinitely many connected 
components $X_m:=X\times \{m\}$, $m\in \ZZ$. The K-theoretic Fock
space is defined by
\ben
\F(X\times \ZZ):= \varinjlim_{k} \ 
\F(X\times [-k,k])= \bigcup_{k=0}^\infty \F(X\times [-k,k]),
\een
where the inclusion $\F(X\times [-k,k])\subset \F(X\times [-k-1,k+1])$
is defined via the pushforward with respect to the natural inclusion
$X\times [-k,k]\subset X\times [-k-1,k+1]$. Theorem
\ref{thm:Heisenberg} implies that 
\beq\label{fock:pol}
\F(X\times \ZZ)\cong 
\CC[\nu_{r,k,\alpha}\ (r\geq 1, k\in \ZZ, 1\leq \alpha\leq N)],
\eeq
where the vacuum is the same as before, that is, $1\in \F(X)=\F(X\times
\{0\}) \subset \F(X\times \ZZ)$ and the creation operator
$\nu_{r,k,\alpha}$ is defined by the induction operation and exterior
tensor product by $(\Phi^{\alpha}_k)^{\boxtimes r}\otimes p_r$ where
$\Phi^\alpha_k$ is a virtual vector bundle on $X\times \ZZ$  whose
restriction to the connected component $X_m$ is $\Phi^\alpha$ for $m=
k$ and $0$ if $m\neq k$.   

The following isomorphism will be needed in the definition of KGW
invariants with values in $\F(X\times \ZZ)$. Suppose
that $Y=Y_1\sqcup\dots \sqcup Y_n$ has finitely many connected
components. We have an isomorphism 
\beq\label{space-ind}
K_{S_k}(Y^k)\cong 
\bigoplus_{b=(b_1, \dots, b_k)}
K_{S_b}(Y_{b_1}\times \cdots \times Y_{b_k}),
\eeq
where the direct sum is over all monotonely increasing sequences
satisfying $1\leq b_1\leq b_2\leq \cdots \leq b_k\leq n$ and
$S_b=\{\sigma\in S_n\ |\ b_{\sigma(i)}=b_i\ \forall i\}$. The
isomorphism is induced by restricting $S_k$-bundles on $Y^k$ to
$Y_{b_1}\times \cdots \times Y_{b_k}$. Note that the inverse of this
isomorphism, that is, constructing an $S_k$-bundle from an
$S_b$-bundles, is given by an operation that in some sense generalizes
the induction operation of vector bundles: first
we have to make an $S_k$-space from an $S_b$-space by adding extra
connected components and then extend the
$S_b$-bundle to an $S_k$-bundle via the usual induction operation.

Applying the isomorphism \eqref{space-ind} to our settings we get
\beq\label{fock-ind}
\F(X\times \ZZ)\cong \bigoplus_{b=(b_1,\dots,b_k)}
K_{S_b}(X_{b_1}\times \cdots \times X_{b_k}),
\eeq
where the direct sum is over all monotonely increasing sequences of
integer numbers $b_1\leq b_2\leq \cdots\leq b_k$, $S_b$ is the (Young)
subgroup of $S_k$ preserving the sequence $b$, and $X_{b_i}=X\times
\{b_i\}$ is the $b_i$-th connected component of $X\times \ZZ$.

Now we are in position to define KGW invariant with values in
$\F(X\times \ZZ)$. The definition \eqref{KGW-inv} should be modified
as follows. We replace the sum over $k$ with a sum over all monotonely
increasing sequences of integers $b=(b_1,\dots,b_k)$, that is,
$b_1\leq \cdots \leq b_k$, define the evaluation map
\ben
\op{ev}^{(b)} :\overline{\M}_{g,n+k}(X,d)\to
X_{b_1}\times \cdots \times X_{b_k},\quad 
(C,p_1,\dots,p_{n+k};f)\mapsto ((f(p_{n+1}),b_1),\dots ,
(f(p_{n+k}),b_k)),
\een
and replace the pushforward by
\ben
\op{ev}^{(b)}_*
\left(
  \O_{g,n+k,d}\otimes L_1^{i_1}\otimes \cdots \otimes L_n^{i_n}\otimes
  L_{n+1}^{b_1}\otimes \cdots \otimes L_{n+k}^{b_k}\otimes 
  \op{ev}_{(n)}^*(\Phi_{a_1}\boxtimes \cdots \boxtimes \Phi_{a_n})
\right).
\een
The above pushforward takes value in $K_{S_b}(X_{b_1}\times \cdots
\times X_{b_k})$. Recalling the isomorphisms \eqref{fock-ind} and
\eqref{fock:pol} we get that the KGW invariants take value in a
certain completion of the polynomial ring
$\CC[\nu_{r,k,\alpha} (r\geq 1, k\in \ZZ,1\leq \alpha\leq N)]$.
Similarly to \eqref{KGW-ti}, we have 
\beq\label{KGW-desc_ti}
\frac{\partial}{\partial \nu_{1,k,\alpha}} \langle
\mathbf{t}(L_1),\dots,\mathbf{t}(L_n)\rangle_{0,n}(\nu)=
\langle
\mathbf{t}(L_1),\dots,
\mathbf{t}(L_n),L_{n+1}^k\Phi_\alpha \rangle_{0,n+1}(\nu).
\eeq
Using the Heisenberg structure of $\F(X\times \ZZ)$, the same argument
as in Section \ref{sec:pec} proves that under the substitution
\ben
\nu_{1,k,\alpha}=t_{1,k,\alpha}-t_{k,\alpha},\quad
\nu_{r,k,\alpha}=\Psi^r(t_{r,k,\alpha})\ (r>1)
\een
the permutation-equivariant genus-g potential  takes the form
\ben
\F^{(g)}(\mathbf{t})=\sum_{n=0}^\infty
\frac{1}{n!} \langle
\mathbf{t}(\psi_1),\dots,\mathbf{t}(\psi_n)\rangle_{g,n} (\nu).
\een

\subsection{String equation}
\label{sec:string}
The following formula holds:
\ben
\langle
\mathbf{t}(L_1),\dots,\mathbf{t}(L_n),1\rangle_{0,n+1}(\nu)=
\langle \mathbf{t}(L_1),\dots,\mathbf{t}(L_n)\rangle_{0,n}(\nu) + \\
\sum_{i=1}^n
\langle
\mathbf{t}(L_1),\dots,
\frac{
  \mathbf{t}(L_i)-\mathbf{t}(1)
}{L_i-1},
\dots \mathbf{t}(L_n)\rangle_{0,n}(\nu) + Q_n(\mathbf{t},\nu),
\een
where $Q_n(\mathbf{t},\nu)=0$ for $n>2$ and
\allowdisplaybreaks[1]
\begin{align*}
Q_2(\mathbf{t},\nu)  & =  (\mathbf{t}(1), \mathbf{t}(1)) \\
  Q_1(\mathbf{t},\nu)  & = (\mathbf{t}(1),\nu_1)\\
  Q_0(\mathbf{t},\nu) & = 
\frac{1}{2}(\nu_1,\nu_1)+\frac{1}{2} (\psi^2(\nu_2),1)
\end{align*}
where $\nu_k=\sum_{\alpha=1}^N \nu_{k,\alpha} \Phi_\alpha$ and
$\psi^m(\nu_k):=\sum_{\alpha=1}^N
\nu_{k,\alpha}\psi^m(\Phi_\alpha)$. The above formula is 
known as the {\em string equation}. The standard technique,
based on the map forgetting the last marked point, works in our
settings too (see \cite{Giv2015Oct}, Proposition 1). We would like to
work out only the correction terms $Q_n(\mathbf{t},\nu)$.  

Suppose that $d=0$ and $n+k=2$ where $k$ is the number of permutable
marked points and $n$ is the number of the remaining ones. There
are 3 cases. First, if $n=2$ and $k=0$, we get
\ben
\langle
\mathbf{t}(L_1),\mathbf{t}(L_2),1\rangle_{0,3,0} =
(\mathbf{t}(1),\mathbf{t}(1))=Q_2( \mathbf{t},\nu).
\een
The second case is when $n=1$ and $k=1$: the moduli space
$\overline{\M}_{0,3}(X,0)= \overline{\M}_{0,3}\times X=X$ and the
corresponding contribution to the KGW invariant is  
\ben
\op{ev}^{(1)}_* (\op{ev}_{(1)}^*(\mathbf{t}(1))
)=\mathbf{t}(1)=Q_1(\mathbf{t},\nu), 
\een
where we used that the topological $K$-ring $K(X)$ is embedded
in the Fock space $\F(X)$ via $v\mapsto \sum_{\alpha=1}^N
(v,\Phi_\alpha)\nu_{1,\alpha}= (v,\nu_1).$
Finally, if $n=0$ and $k=2$: the evaluation map $\op{ev}^{(2)}$
coincides with the diagonal embedding $\Delta: X\to X\times X$. We
have to compute $f(\nu_1,\nu_2):=\Delta_*(\O_X)$.  Let us compute the
derivatives of $f$ with respect to $\nu_1$ and $\nu_2$.
\ben
\partial_{\nu_{1,\alpha}} f= \pi^{(2)}_*(
(\Phi_\alpha\boxtimes 1)\otimes \Delta_*(\O_X) )=
\pi^{(2)}_*\Delta_* (\Phi_\alpha)=\Phi_\alpha =
\sum_{\beta=1}^N (\Phi_\alpha,\Phi_\beta) \nu_{1,\beta},
\een
where $\pi^{(2)}:X\times X\to X$ is the projection on the second
factor. Similarly,
\ben
2\partial_{\nu_{2,\alpha}} f =\op{tr}_{(12)} \pi_*(
(\Phi_\alpha\boxtimes \Phi_\alpha )\otimes \Delta_*(\O_X)) =
\op{tr}_{(12)} 
\pi_*\Delta_*(\Phi_\alpha^{\otimes 2})=
\pi_*\Delta_*\op{Tr}_{(12)}  (\Phi_\alpha^{\otimes
  2})=(\psi^2(\Phi_\alpha),1), 
\een
where $\pi:X\times X\to {\rm pt}$ is the constant map and $\pi_*$ is
the $K$-theoretic pushforward. We get
\ben
\Delta_*(\O_X)=
\frac{1}{2}\sum_{\alpha,\beta=1}^N
\chi(\Phi_\alpha\otimes \Phi_\beta)\,  \nu_{1,\alpha}\nu_{1,\beta} +
\frac{1}{2}\sum_{\alpha=1}^N
\chi(\psi^2(\Phi_\alpha))\, \nu_{2,\alpha}=Q_0(\mathbf{t},\nu).
\een

\subsection{Dilaton equation}
\label{sec:dilaton}

The following formula is the permutation-equivariant K-theoretic version of the so called {\em dilaton equation}:
\ben
\langle \mathbf{t}(L_1),\dots,\mathbf{t}(L_n), L_{n+1}-1
\rangle_{0,n+1}(\nu) = 
\Big(n-2+\sum_{b=1}^N \nu_{1,b}\partial_{\nu_{1,b}}\Big)
\langle \mathbf{t}(L_1),\dots,\mathbf{t}(L_n)
\rangle_{0,n}(\nu) - \delta_{n,0} (\psi^2(\nu_2),1),
\een
where $\psi^2(\nu_2)=\sum_{\alpha=1}^N \nu_{2,\alpha} \psi^2(\Phi_\alpha)$. Using formula \eqref{KGW-ti} we can rewrite the RHS of the above formula as follows:
\ben
(n-2)\langle \mathbf{t}(L_1),\dots,\mathbf{t}(L_n)
\rangle_{0,n}(\nu)+
\langle \mathbf{t}(L_1),\dots,\mathbf{t}(L_n),\nu_1
\rangle_{0,n+1}(\nu)- \delta_{n,0} (\psi^2(\nu_2),1).
\een
The standard argument proving the dilaton equation in cohomological Gromov--Witten theory is based on pushing forward along the forgetful map 
$\overline{\M}_{0,n+1+k}(X,d)\to \overline{\M}_{0,n+k}(X,d)$. The
K-theoretic version of the proof works in our settings too. We refer
to \cite{Giv2015Oct} for more details. We would like only to work out
the exceptional term corresponding to the case when $d=0$ and $n+k=2$,
that is, the case when the forgetful map does not exist. If $k<2$,
then $L_{n+1}$ is trivial as an $S_k$-bundle, so the exceptional term
is $0$. The only non-trivial contribution comes when $n=0$ and $k=2$.
We have to compute $f(\nu):=\op{ev}^{(2)}_*(L_1-1)=
\Delta_*(\O_X\otimes (\epsilon-1))$, where $\Delta:X\to X\times X$ is
the diagonal embedding and $\epsilon=L_1$ is the alternating
representation of $S_2$. We have  
\ben
2\partial_{\nu_{\alpha,2}} f = \op{tr}_{(12)} \pi_* \left(
\Phi_\alpha^{\boxtimes 2} \Delta_*(\O_X\otimes (\epsilon-1))\right) =
(\pi\circ\Delta)_* (\op{Tr}_{(12)} (\Phi_\alpha^{\otimes 2}\otimes (\epsilon-1)))=-2(\psi^2(\Phi_\alpha),1).
\een 
Similar computation, since $\op{Tr}_1(\epsilon-1)=0$, proves that $\partial_{\nu_{a,1}}f =0$. Therefore, 
\ben
f(\nu)=-\sum_{a=1}^N \nu_{a,2} (\psi^2(\Phi_a),1)=-(\psi^2(\nu_2),1).
\een

\subsection{WDVV equations and the $S$-matrix}
\label{sec:wdvv}

The WDVV equations say that the expression  
\beq\label{WDVV-1234}
\sum_{a,b=1}^N G^{ab}(\nu)
\left\langle \frac{\phi_1}{1-x_1L},\frac{\phi_2}{1-x_2L},\Phi_a
\right\rangle_{0,3}(\nu)
\left\langle \frac{\phi_3}{1-x_3L},\frac{\phi_4}{1-x_4L},\Phi_b
  \right\rangle_{0,3}(\nu)
\eeq
is symmetric in the pairs $(\phi_i,x_i)$ ($1\leq i\leq 4$). For
example, by exchanging $(\phi_2,x_2)$ and $(\phi_3,x_3)$, we get 
\beq\label{WDVV-1324}
\sum_{a,b=1}^N G^{ab}(\nu)
\left\langle \frac{\phi_1}{1-x_1L},\frac{\phi_3}{1-x_3L},\Phi_a\right\rangle_{0,3}(\nu)
\left\langle \frac{\phi_2}{1-x_2L},\frac{\phi_4}{1-x_4L},\Phi_b\right\rangle_{0,3}(\nu).
\eeq
The equality of \eqref{WDVV-1234} and \eqref{WDVV-1324} determines the
remaining identities.  The standard proof based on the forgetful map
$\overline{\M}_{0,4+k}(X,d)\mapsto \overline{\M}_{0,4}=\PP^1$ works
in our settings too (see \cite{Giv2015Oct}, Proposition 3). The only new feature, compared
to cohomological GW theory, is the metric \eqref{metric:G} -- see \cite{Giv2000}
for more details.

Let us recall the operator $S(\nu,q)\in \op{End}(K(X))$ depending on
the complex number $q\in \CC^*$ and the parameters
$\nu=(\nu_{r,\alpha})$ -- see formula \eqref{S-matrix}. 
\begin{proposition}\label{prop:S-matrix}
a) The following formula holds:
\ben
G(S(\nu,q_1)\Phi_a,S(\nu,q_2)\Phi_b) =
(\Phi_a,\Phi_b)+(1-q_1^{-1}q_2^{-1})
\Big\langle
\frac{\Phi_a}{1-q_1^{-1}L}, \frac{\Phi_b}{1-q_2^{-1}L}
\Big\rangle_{0,2}(\nu).
\een

b) The following formula holds: $S^T(\nu,q^{-1}) S(\nu,q)=1$ where $S^T$ is defined by the identity $G(Sa,b)=(a,S^Tb)$. 

c) The matrix $S$ satisfies the following differential equations:
\ben
(q-1)\, \partial_{\nu_{1,i}} S(\nu,q)= \Phi_i\bullet S(\nu,q),\quad 
1\leq i\leq N,
\een
where $\bullet$ is the permutation-equivariant quantum K-product
defined by \eqref{K-quantum}.
\end{proposition}
\proof
a) Let us apply the WDVV equations with $\phi_1=\phi_3=1$, $x_1=x_3=0$. For brevity, put $\phi=\phi_2$, $x=x_2$ and $\psi=\phi_4$, $y=x_4$. We get 
\ben
\Big\langle
1,\frac{\phi}{1-xL},\Phi_a
\Big\rangle_{0,3} 
\Big\langle
1,\frac{\psi}{1-yL},\Phi_b
\Big\rangle_{0,3}\, G^{ab} =
\Big\langle
1,1,\Phi_a
\Big\rangle_{0,3} 
\Big\langle
\frac{\phi}{1-xL},\frac{\psi}{1-yL},\Phi_b
\Big\rangle_{0,3}\, G^{ab},
\een
where we use Einstein's convention to sum over repeating upper and lower indexes. Using the string equation, we get $\Big\langle
1,1,\Phi_a
\Big\rangle_{0,3} = G(1,\Phi_a)$. The RHS takes the form 
\ben
\Big\langle
\frac{\phi}{1-xL},\frac{\psi}{1-yL},1
\Big\rangle_{0,3}= 
\Big(1+\frac{x}{1-x}+\frac{y}{1-y}\Big)\,
\Big\langle
\frac{\phi}{1-xL},\frac{\psi}{1-yL}
\Big\rangle_{0,2} + 
\frac{(\phi,\psi)}{(1-x)(1-y)},
\een
where we used the string equation and the identity
\ben
\frac{1}{L-1}\Big(\frac{1}{1-xL}-\frac{1}{1-x}\Big)=
\frac{x}{1-x}\,
\frac{1}{1-xL}.
\een
The RHS of the WDVV equation takes the form 
\ben
\frac{1}{(1-x)(1-y)}\Big(
(\phi,\psi)+
(1-xy) \Big\langle
\frac{\phi}{1-xL},\frac{\psi}{1-yL}
\Big\rangle_{0,2} 
\Big).
\een
Similarly, using the string equation we get that 
\ben
\Big\langle
1,\frac{\phi}{1-xL},\Phi_a
\Big\rangle_{0,3} =\frac{1}{1-x}\, 
\Big(
\Big\langle
\frac{\phi}{1-xL},\Phi_a
\Big\rangle_{0,3} +(\phi,\Phi_a)\Big)=
\frac{1}{1-x}\,
G(S(\nu,x^{-1})\phi,\Phi_a).
\een
Now it is clear that 
\ben
\frac{1}{(1-x)(1-y)} \, G(S(\nu,x^{-1})\phi,S(\nu,y^{-1})\psi)=
\frac{1}{(1-x)(1-y)}\Big(
(\phi,\psi)+
(1-xy) \Big\langle
\frac{\phi}{1-xL},\frac{\psi}{1-yL}
\Big\rangle_{0,2} 
\Big) 
\een
and this is precisely the identity that we had to prove. 

b) is a direct consequence of a). 

c) Let us apply the WDVV equations with 
\ben
\phi_1=1,\ x_1=0,\quad 
\phi_2=\phi,\ x_2=x,\quad
\phi_3=\Phi_i,\ x_3=0,\quad
\phi_4=\Phi_j,\ x_4=0.
\een 
We get 
\ben
\Big\langle
1,\frac{\phi}{1-xL},\Phi_a\Big\rangle \, 
\langle \Phi_i,\Phi_j,\Phi_b\rangle\ G^{ab} =
\langle
1,\Phi_i,\Phi_a\rangle \, 
\Big\langle \frac{\phi}{1-xL},\Phi_j,\Phi_b\Big\rangle\ G^{ab}
\een
Again, using the string equation, we get that $\langle
1,\Phi_i,\Phi_a\rangle_{0,3} =G_{ia}$ and that the RHS becomes 
\ben
\Big\langle \frac{\phi}{1-xL},\Phi_j,\Phi_i\Big\rangle_{0,3}=
\partial_{\nu_{1,i}} G(S(\nu,x^{-1})\phi,\Phi_j)=
(\phi, \partial_{\nu_{1,i}} \, S^T(\nu,x^{-1}) \Phi_j).
\een
On the other hand, the LHS (see the computation in part a)) after removing $1$ via the string equation and recalling the definition of the quantum $K$-product, becomes 
\ben
\frac{1}{1-x} G(S(\nu,x^{-1})\phi,\Phi_i\bullet \Phi_j)=
\frac{1}{1-x} (\phi, S^T(\nu,x^{-1})\Phi_i\bullet \Phi_j).
\een
Therefore, 
\ben
\partial_{\nu_{1,i}} \, S^T(\nu,x^{-1}) = \frac{1}{1-x}\, 
S^T(\nu,x^{-1})\Phi_i\bullet .
\een
Recalling b), we get $S^T(\nu,x^{-1})=S(\nu,x)^{-1}$ and the formula that we have to prove becomes a direct consequence of the above formula.
\qed

\section{Genus-0 integrable hierarchies}

The goal of this section is to prove Theorem \ref{thm:top_sol}. The
argument is essentially the same as in \cite{MT2018}, Section 4.3.  

\subsection{From descendants to ancestors}

The {\em ancestor} KGW invariants 
$\langle
\mathbf{t}(\overline{L}_1),\dots,
\mathbf{t}(\overline{L}_n)\rangle_{g,n}(\nu)$ are defined in the same
way as the descendent KGW invariants, that is, by formula
\eqref{KGW-inv}, except that we replace $L_i$ with the pullback
$\overline{L}_i=\op{ft}^*L_i$ where $\op{ft}:
\overline{\M}_{g,n+k}(X,d)\to \overline{\M}_{g,n}$ is the map that
forgets the stable map and the last $k$ marked points. According to
Givental (see \cite{Giv2015Oct}), the ancestor invariants can be be
expressed in terms of the descendent ones and the $S$-matrix. Let us
recall Givental's formula.  

Let $\CC(q)$ be the field of all rational functions on $\CC$. Using elementary fraction decomposition, we have a natural projection map: 
\beq\label{ef-projection}
[\ ]_+: \CC(q)\to \CC[q,q^{-1}] ,\quad 
f(q)\mapsto [f(q)]_+ := 
-\op{Res}_{w=0,\infty} \ \frac{f(w)dw}{w-q}.
\eeq
It is easy to check that the above residue truncates all terms in the elementary fraction decomposition of $f(q)$ that have poles in $\CC^*=\CC\setminus{\{0\}}$. The matrix $S(\nu,q)$ has entries that are formal power series in $\mathbf{t}$, $\nu$, and the Novikov variables $Q$ whose coefficients are rational functions. In particular, we can define $[S\mathbf{t}]_+(\nu,q):=[S(\nu,q)\mathbf{t}(q)]_+$.
\begin{lemma}\label{le:S-rotation}
a) The following formula holds:
\ben
[S\mathbf{t}]_+(\nu,q)= \mathbf{t}(q) + 
\sum_{\alpha,\beta=1}^N 
\Big\langle
\frac{\mathbf{t}(L)- \mathbf{t}(q)}{L-q}\, L, \Phi_\alpha
\Big\rangle_{0,2}(\nu)\, G^{\alpha\beta}(\nu)\, \Phi_\beta. 
\een
b) The following formula holds:
\ben
[S\mathbf{t}]_+(\nu,1)= 
\sum_{\alpha,\beta=1}^N 
\Big\langle1,\Phi_\alpha,\mathbf{t}(L)
\Big\rangle_{0,3}(\nu)\, G^{\alpha\beta}(\nu)\, \Phi_\beta. 
\een
\end{lemma}
\proof
a)
Recalling the definition of the $S$-matrix we get 
\ben
G([S\mathbf{t}]_+(\nu,q),\Phi_\alpha)= (\mathbf{t}(q),\Phi_\alpha) + 
\Big\langle
-\op{Res}_{w=0,\infty} \frac{\mathbf{t}(w) dw}{(1-w^{-1} L)(w-q)},\Phi_\alpha
\Big\rangle_{0,2}(\nu).
\een
Recalling Cauchy's theorem we get 
\ben
-\op{Res}_{w=0,\infty} \frac{\mathbf{t}(w) dw}{(1-w^{-1} L)(w-q)}=
\op{Res}_{w=L,q} \frac{\mathbf{t}(w) dw}{(1-w^{-1} L)(w-q)} =
\frac{\mathbf{t}(L)L-\mathbf{t}(q)q}{L-q}=
\frac{\mathbf{t}(L)-\mathbf{t}(q)}{L-q}\, L + \mathbf{t}(q).
\een
It remains only to recall the definition of $G$ and to note that
\ben
(\mathbf{t}(q),\Phi_\alpha)+
\langle \mathbf{t}(q), \Phi_\alpha\rangle_{0,2}(\nu)= 
G(\mathbf{t}(q),\Phi_\alpha).
\een
Part b) follows from a) and the string equation 
\ben
\Big\langle 1,\Phi_\alpha,\mathbf{t}(L)
\Big\rangle_{0,3}(\nu)=
\langle \mathbf{t}(L),\Phi_\alpha\rangle_{0,2}(\nu) + 
\Big\langle
\frac{\mathbf{t}(L)- \mathbf{t}(1)}{L-1},\Phi_\alpha\Big\rangle_{0,2}(\nu)+
(\mathbf{t}(1),\Phi_\alpha).\qed
\een
The relation between descendants and ancestors is given by the following formula:
\beq\label{desc-anc}
\langle 
\mathbf{t}(L_1),\dots, \mathbf{t}(L_n)\rangle_{g,n}(\nu)=
\langle
[S\mathbf{t}]_+(\nu,\overline{L}_1),\dots, 
[S\mathbf{t}]_+(\nu,\overline{L}_n)\rangle_{g,n}(\nu),
\eeq
where $g$ and $n$ must satisfy $2g-2+n>0$. 

Let us choose a formal power series $\tau(\mathbf{t},\nu_2,\nu_3,\dots)$ with coefficients in $K(X)$, such that,
\beq\label{fixed-point-eq}
[S\mathbf{t}]_+(\tau,\nu_2,\nu_3,\dots,1)=\tau.
\eeq
This is a fixed-point problem for $\tau$ and we can construct a formal solution via the iterations: 
\ben
\tau^{(0)}:=0,\quad 
\tau^{(n+1)}:= [S\mathbf{t}]_+(\tau^{(n)},\nu_2,\nu_3,\dots,1).
\een 
The sequence $\tau^{(n)}$ of formal series has a limit as $n\to
\infty$ which provides a solution to our fixed-point problem.
\begin{lemma}\label{le:unstable_qf}
  The following formula holds:
  \ben
  \langle\ \rangle_{0,0}(\nu)+\frac{1}{2}(\psi^2(\nu_2),1)+
  \langle\mathbf{t}\rangle_{0,1}(\nu)+\frac{1}{2}
  \langle \mathbf{t},\mathbf{t}\rangle_{0,2}(\nu) =
  \frac{1}{2}\langle \mathbf{t}+1-L+\nu_1, \mathbf{t}+1-L+\nu_1\rangle_{0,2}(\nu),
  \een
  where we suppressed the dependence of $\mathbf{t}(L)$ on $L$, that
  is, $\mathbf{t}=\mathbf{t}(L)$.   
\end{lemma}
\proof
The quadratic terms in $\mathbf{t}$ clearly match. The equality of the
linear terms is equivalent to
\ben
\langle\mathbf{t} \rangle_{0,1}(\nu) = \langle\mathbf{t} ,1-L+\nu_1\rangle_{0,2}(\nu).
\een
The above equality is precisely the dilaton equation 
$\langle\mathbf{t},L-1\rangle_{0,2}(\nu)=-
\langle\mathbf{t}\rangle_{0,1}(\nu)+
\langle\mathbf{t},\nu_1\rangle_{0,2}(\nu)
$. The equality of the
remaining terms is equivalent to
\ben
\langle\ \rangle_{0,0}(\nu) +\frac{1}{2}(\psi^2(\nu_2),1) =\frac{1}{2}
\langle 1-L+\nu_1, 1-L+\nu_1\rangle_{0,2}(\nu).
\een
The RHS can be written as
\beq\label{constant_term}
\frac{1}{2}
\langle L-1, L-1\rangle_{0,2}(\nu)-
\langle L-1,\nu_1\rangle_{0,2}(\nu)+\frac{1}{2}
\langle\nu_1,\nu_1\rangle_{0,2}.
\eeq
Using the dilaton equation we get
\ben
\langle L-1,\nu_1\rangle_{0,2}(\nu)=-
\langle\nu_1\rangle_{0,1}(\nu) + 
\langle \nu_1,\nu_1\rangle_{0,2}(\nu)
\een
and
\ben
\langle L-1, L-1\rangle_{0,2}(\nu) =
-\langle L-1\rangle_{0,1}(\nu) + \langle L-1, \nu_1\rangle_{0,2}(\nu)=
2\langle\ \rangle_{0,0}(\nu) -
2\langle \nu_1 \rangle_{0,1}(\nu) + (\psi^2(\nu_2),1)+
\langle \nu_1, \nu_1\rangle_{0,2}(\nu).
\een
Substituting the above two formulas in \eqref{constant_term} we get
the identity that we had to prove.\qed
\begin{proposition}\label{prop:reconstr}
Let
\ben
\F(\mathbf{t}):= \frac{1}{2}(\psi^2(\nu_2),1)  +
\sum_{n=0}^\infty \frac{1}{n!}
\langle \mathbf{t}(L),\dots,\mathbf{t}(L)\rangle_{0,n}(0,\nu_2,\nu_3,\dots)
\een
and let $\tau(\mathbf{t},\nu_2,\nu_3,\dots)$ be a solution to the
fixed-point problem \eqref{fixed-point-eq}. Then the following
formulas hold:
\ben
  \F(\mathbf{t})  = \frac{1}{2}
                   \langle \mathbf{t}(L)+1-L,
                   \mathbf{t}(L)+1-L\rangle_{0,2}(\tau,\nu_2,\nu_3,\dots),
\een
\ben
  \partial_{t_{m,\alpha}} \F(\mathbf{t}) = \frac{1}{2}
                   \langle \Phi_\alpha (L-1)^m,
                   \mathbf{t}(L)+1-L\rangle_{0,2}(\tau,\nu_2,\nu_3,\dots),
\een
and
\ben
  \partial_{t_{m,\alpha}}\partial_{t_{n,\beta}} \F(\mathbf{t}) = \frac{1}{2}
                   \langle \Phi_\alpha (L-1)^m, \Phi_\beta (L-1)^n
                   \rangle_{0,2}(\tau,\nu_2,\nu_3,\dots).
\een
\end{proposition}
\proof
Let us prove the first formula only. The argument for the remaining
two is similar. Using the translation invariance \eqref{KGW-ti}, we
can rewrite $\F$ as follows:
\beq\label{shifted-F}
\F(\mathbf{t}):= \frac{1}{2}(\psi^2(\nu_2),1)  +
\sum_{n=0}^\infty \frac{1}{n!}
\langle \mathbf{t}-\nu_1,\dots,\mathbf{t}-\nu_1\rangle_{0,n}(\nu_1,\nu_2,\nu_3,\dots)
\eeq
where $\nu_1$ is an arbitrary formal parameter and we wrote
$\mathbf{t}$ instead of $\mathbf{t}(L)$ in order to make the notation
less cumbersome. Let us apply the formula expressing descendants in
terms of ancestors \eqref{desc-anc}. We get 
\ben
\langle
\mathbf{t}-\nu_1,\dots,\mathbf{t}-\nu_1\rangle_{0,n}(\nu_1,\nu_2,\nu_3,\dots)=
\langle
[S\mathbf{t}]_+(\nu,\overline{L}_1)-\nu_1,\dots,
[S\mathbf{t}]_+(\nu,\overline{L}_n)-\nu_1
\rangle_{0,n}(\nu_1,\nu_2,\nu_3,\dots),
\een
where $n\geq 3$ and we used Lemma \ref{le:S-rotation} to compute
$[S\nu_1]_+=\nu_1$. Let us specialize
$\nu_1=\tau(\mathbf{t},\nu_2,\nu_3,\dots)$. By definition
$[S\mathbf{t}]_+(\nu,1)-\nu_1=0$ which implies that
$[S\mathbf{t}]_+(\nu,\overline{L}_i)-\nu_1$ is proportional to
$\overline{L}_i-1=\op{ft}^*(L_i-1)$ for all $1\leq i\leq n$. However,
the product $(L_1-1)\cdots (L_n-1)=0$ on $\overline{\M}_{0,n}$ because
$\overline{\M}_{0,n}$ is an $(n-3)$-dimensional manifold.  We conclude
that after the substitution $\nu_1=\tau(\mathbf{t},\nu_2,\nu_3,\dots)$
all terms in the sum in \eqref{shifted-F} with $n\geq 3$ must
vanish. The sum of the remaining terms, according to Lemma
\ref{le:unstable_qf}, must be
\ben
\frac{1}{2}\langle
\mathbf{t}-\nu_1 + 1-L +\nu_1,
\mathbf{t}-\nu_1 + 1-L +\nu_1\rangle_{0,2}(\nu)=
\frac{1}{2}\langle
\mathbf{t} + 1-L,
\mathbf{t} + 1-L \rangle_{0,2}(\nu).\qed
\een
\subsection{Proof of Theorem \ref{thm:top_sol}}

Let $\tau(\mathbf{t},\nu_2,\nu_3,\dots)$ be the solution to the
fixed-point problem \eqref{fixed-point-eq}. We will prove Theorem
\ref{thm:top_sol} by showing that
$v(\mathbf{t})=\tau(\mathbf{t},\nu_2,\nu_3,\dots)$ and that
$\tau(\mathbf{t},\nu_2,\nu_3,\dots)$ is a solution to the integrable
hierarchy \eqref{KGW:ph}. Let us denote the partial derivative
$\partial_{t_{m,\alpha}}$ by $\partial_{m,\alpha}$.

Let us prove that $v(\mathbf{t})=\tau(\mathbf{t},\nu_2,\nu_3,\dots)$,
that is,
\beq\label{J-equation}
J(\tau,0)=1+\sum_{\alpha=1}^N
\partial_{{0,\alpha}} \partial_{{0,1}} \F(\mathbf{t}) \Phi^\alpha,
\eeq
where $\F(\mathbf{t})$ is the same as in Proposition \ref{prop:reconstr} and
recall that $\Phi_1=1$. According to Proposition \ref{prop:reconstr},
the second order partial derivative in the above formula is
\ben
\langle 1, \Phi_\alpha\rangle_{0,2}(\tau,\nu_2,\nu_3,\dots)=
\langle \Phi_\alpha\rangle_{0,1}(\tau,\nu_2,\nu_3,\dots) + (\Phi_\alpha,\tau),
\een
where we used the string equation. The formula that we want to prove
follows.

Let us prove that $\tau$ is a solution to
\eqref{KGW:ph}. Differentiating \eqref{J-equation} with respect to
$t_{n,\beta}$ we get 
\beq\label{J-deriv}
\partial_{n,\beta} J(\tau,0)= \partial_{0,1} \sum_{\alpha=1}^N
\langle \Phi_\alpha, \Phi_\beta (L-1)^n\rangle_{0,2}(\tau)\, \Phi^\alpha,
\eeq
where we suppressed the dependence of the two-point correlator on
$\nu_2,\nu_3,\dots$ and we used Proposition \ref{prop:reconstr}. Using
the projection map \eqref{ef-projection} we transform the
two-point correlator  $\partial_{0,1} 
\langle \Phi_\alpha, \Phi_\beta (L-1)^n\rangle_{0,2}(\tau)$ into
\ben
-\partial_{0,1} \op{Res}_{q=0,\infty} dq (q-1)^n
\Big\langle
\Phi_\alpha, \frac{\Phi_\beta}{q-L}
\Big\rangle (\tau)=
-\partial_{0,1} \op{Res}_{q=0,\infty} dq  q^{-1} (q-1)^n
G(S(\tau,q)\Phi_\beta,\Phi_\alpha).
\een
Note that $G(a,S(\tau,0)b)= (a,b)$, or equivalently
$G(a,b)=(a,S(\tau,0)^{-1}b)$. Let us change
$G(S(\tau,q)\Phi_\beta,\Phi_\alpha)$ into
$(S(\tau,0)^{-1}S(\tau,q)\Phi_\beta,\Phi_\alpha)$ and substitute the
result in formula \eqref{J-deriv}. We get
\beq\label{J-deriv_2}
\partial_{n,\beta} J(\tau,0) = -\partial_{0,1} \op{Res}_{q=0,\infty}
dq  q^{-1} (q-1)^n S(\tau,0)^{-1} S(\tau,q)\Phi_\beta. 
\eeq
Using that $S(\nu,q)$ is a solution to the quantum differential
equations and that $J(\nu,q)=(1-q)S(\nu,q)^{-1} 1$ we get
\ben
\partial_{n,\beta} J(\tau,0) = -S(\tau,0)^{-1} (-\partial_{n,\beta}\tau
\bullet S(\tau,0) )  S(\tau,0)^{-1}1 = S(\tau,0)^{-1} \partial_{n,\beta}\tau.
\een
Similarly, we find that $\partial_{0,1} S(\tau,0)^{-1} S(\tau,q) $ is 
\ben
S(\tau,0)^{-1}
\partial_{0,1}\tau\bullet S(\tau,q) + S(\tau,0)^{-1} (q-1)^{-1}
\partial_{0,1}\tau\bullet S(\tau,q) =\frac{q}{q-1} S(\tau,0)^{-1}
\partial_{0,1}\tau \bullet S(\tau,q).
\een
Substituting these formulas in \eqref{J-deriv_2} and cancelling
$S(\tau,0)^{-1}$ we get
\ben
\partial_{n,\beta} \tau =
-\op{Res}_{q=\infty}
dq (q-1)^{n-1} \partial_{0,1}\tau \bullet S(\tau,q) \Phi_\beta,
\een
where the residue at $q=0$ was dropped because $S(\tau,q)$ does not
have a pole at $q=0$. Comparing with \eqref{KGW:ph}, we get that
$\tau(\mathbf{t},\nu_2,\nu_3,\dots)$ is a solution to the principal
hierarchy. \qed

\section{The genus-0 K-theoretic Gromov--Witten invariants of the point}

The goal of this section is to prove Theorem \ref{thm:corr-0}.

\subsection{Genus-0 moduli spaces of curves}
Note that the vector space $V_n=\{a\in \CC^n\ |\ a_1+\cdots+a_n=0\}$ is naturally a $S_n$-space where $S_n$ acts by permuting the coordinates
$\sigma\cdot
(a_1,\dots,a_n):=(a_{\sigma^{-1}(1)},\dots,a_{\sigma^{-1}(n)})$.
We have an isomorphism $H^0(\overline{\M}_{0,n+1},L_{n+1})\cong V_n$
defined as follows. Given $a\in V_n$ and $(C,p_1,\dots,p_{n+1})\in \overline{\M}_{0,n+1}$ we 
construct a cotangent vector $\varphi_a(p_{n+1})\in T^*_{p_{n+1}}C$
where the meromorphic 1-form $\varphi_a$ on $C$ is uniquely
determined by the following conditions:
\begin{enumerate}
\item[(i)] $\varphi_a$ is holomorphic except for at most first order
  poles at the nodes and the first $n$ marked points of $C$.
\item[(ii)]
  If $q\in C$ is a node and $C'$ and  $C''$ are the two irreducible
  components of $C$ at $q$, then
  $\op{Res}_{p'=q}(\varphi_a(p'))+\op{Res}_{p''=q}(\varphi_a(p''))=0$
  where $p'$ and $p''$ are local coordinates on respectively $C'$ and
  $C''$ near $q$. 
\item[(iii)]
  $\op{Res}_{p=p_i}\varphi_a(p)=a_i$. 
\end{enumerate}
Holomorphic forms satisfying conditions (i) and (ii) are by
definitions sections of the sheaf $\omega_C(p_1+\cdots+p_n)$ where
$\omega_C$ is the dualizing sheaf of $C$. Condition (iii), uniquely
determines the meromorphic 1-form $\varphi_a$  because $C$ is rational
and hence the only obstruction to the existence of $\varphi_a$ is that the
sum of the residues at all poles is $0$, that is, $a\in V_n$. This
description also shows that $L_{n+1}\cong \omega_C(p_1+\cdots+p_n)$
which is a standard fact.

Recall that every line bundle, provided it has sufficiently many global sections, determines a holomorphic map to a projective space (see \cite{GH1994},
Chapter I, Section 4). In our case, 
the line bundle $L_{n+1}$ determines a holomorphic map
$\pi: \overline{\M}_{0,n+1}\to \PP(V_n^*)$, such that, $\pi^*\O(1)=L_{n+1}$. More precisely, we have 
\beq\label{pi-blowdown}
\pi(C,p_1,\dots,p_{n+1})(a):= \varphi_a(p_{n+1})/dw(p_{n+1}),
\eeq
where $w$ is a local coordinate on $C$ at $p_{n+1}$. The class of $\pi(C,p_1,\dots,p_{n+1})$ in $\PP(V_n^*)$ is independent of the choice of local coordinate $w$.
\begin{lemma}
Suppose that the action $S_n$ on $\PP(V_n^*)$ is induced from $(\sigma\xi)(a):=\xi(\sigma^{-1} a)$ where $\xi\in V_n^*$ and $a\in V_n$. The map \eqref{pi-blowdown} is $S_n$-equivariant, that is, 
$\pi (\sigma (C,p_1,\dots,p_n))=\sigma \pi(C,p_1,\dots,p_n),$ 
where the action of $S_n$ on $\overline{\M}_{0,n+1}$ is defined by  
$\sigma(C,p_1,\dots,p_{n+1}):= 
(C,p_{\sigma^{-1}(1)},\dots, p_{\sigma^{-1}(n)},p_{n+1} )$.  
\end{lemma}
\proof
Let $\psi_a$ be the meromorphic form on $C$ corresponding to $(C,p_{\sigma^{-1}(1)},\dots, p_{\sigma^{-1}(n)},p_{n+1} )$. We have $\op{Res}_{p=p_{\sigma^{-1}(i)}}\varphi_a(p) = a_i$ or equivalently 
\ben
\op{Res}_{p=p_{j}}\psi_a(p) = a_{\sigma(j)}=(\sigma^{-1}a)_j=
\op{Res}_{p=p_j} \varphi_{\sigma^{-1}a}(p).
\een
Therefore, $\psi_a=\varphi_{\sigma^{-1}a}$ and the equivariance of $\pi$ is clear.
\qed

The map $\pi$ can be described more explicitly as follows. Suppose that we have a point $(C,p_1,\dots,p_{n+1})\in \overline{\M}_{0,n+1}$. Let $C_0\cong \PP^1$, we call it the {\em central component}, be the irreducible component of $C$ that carries the marked point $p_{n+1}$. Let us denote by $q_1,\dots q_l$ the nodal points of $C$ that are on $C_0$. Put $I:=\{i\in \{1,2,\dots,n\}\ |\ p_i\in C_0\}$. Note that by removing $C_0$ from $C$ we get a curve consisting of $l$ connected components $C_1,\dots,C_l$. Therefore, the remaining marked points $\{1,2,\dots,n\}\setminus I$ split into $l$ pairwise disjoint groups $J_1,J_2,\dots,J_l$, that is, $J_k=\{j\ |\ p_j\in C_k\}$. Since the isomorphism $C_0\cong \PP^1$ is determined up to the action of $\op{PSL}(2,\CC)$, we can arrange that $p_{n+1}=\infty$ and 
$|J_1| q_1+\dots + |J_l| q_l +\sum_{i\in I} p_i=0$ where $|J_k|$ is the number of elements in $J_k$. Note that
\ben
\varphi_a|_{C_0}= \sum_{k=1}^l \sum_{j\in J_k} \frac{a_j dz}{z-q_k} +
\sum_{i\in I} \frac{a_i dz}{z-p_i}.
\een 
Let $w=1/z$ be the coordinate near $\infty=p_{n+1}$. We have 
\ben
\frac{\varphi_a(w)}{dw}= -\sum_{k=1}^l \sum_{j\in J_k}
\frac{a_jq_k}{1-w q_k} -
\sum_{i\in I} \frac{a_i p_i}{1-w p_i}.
\een
Restricting $w=0$ we get that 
\ben
\pi(C,p_1,\dots,p_{n+1})(a)= -\sum_{k=1}^l \sum_{j\in J_k}
a_jq_k  -
\sum_{i\in I} a_i p_i.
\een
On the other hand, the projective space $\PP(V_n^*)\cong \{\xi\in \PP^{n-1}\ |\ \xi_1+\cdots +\xi_n=0\}$ where $(\xi_1:\xi_2:\cdots:\xi_n)$ are the homogeneous coordinates on $\PP^{n-1}$ and the isomorphism is given by mapping $(\xi_1:\xi_2:\cdots:\xi_n)$ to the linear functional $a\mapsto \xi_1 a_1+\cdots +\xi_n a_n$. Under this identification, the action of $S_n$ on $\PP^{n-1}$ is given by $\sigma (\xi_1:\xi_2:\cdots:\xi_n)=(\xi_{\sigma^{-1}(1)} :\xi_{\sigma^{-1}(2)}:\cdots:\xi_{\sigma^{-1}(n)} )$ and 
\beq\label{pi-blowdown_2}
\pi(C,p_1,\dots,p_{n+1})= (p_i(i\in I): \underbrace{q_k:\cdots :q_k}_{|J_k|}(1\leq k\leq l) ),
\eeq
where the homogeneous coordinates on the RHS of the above formula
should be ordered as follows: $p_i$ ($i\in I$) is placed on the $i$-th
position and a copy of $q_k$ is placed on the $j$-th position for every
$j\in J_k$.  
Using formula \eqref{pi-blowdown_2}, let us describe the fibers of $\pi$. Given $\xi=(\xi_1:\cdots:\xi_n)$ we define an equivalence relation in $\{1,2,\dots,n\}$ by saying that $i\sim j$ iff $\xi_i=\xi_j$. This equivalence relation splits $\{1,2,\dots,n\}$ into several equivalence classes. The ones consisting of a single element we denote by $\{i\}$ ($i\in I$). The remaining ones, consisting of at least two elements, we denote be $J_1,\dots,J_l$. If $(C,p_1,\dots,p_{n+1})$ is in the fiber $\pi^{-1}(\xi)$ then the central component $C_0$ must be $\PP^1$ with marked points $p_i$ ($i\in I$) and nodal points $q_k=\xi_j$ ($1\leq k\leq l$) where for each $k$ we choose one $j\in J_k$. In other words, the isomorphism class of the central component $C_0$ is uniquely fixed. The remaining $l$ components can be fixed in $\overline{\M}_{0,|J_1|+1} \times \cdots \times \overline{\M}_{0,|J_l|+1}$ ways, that is, the fiber 
\beq\label{pi:fibers}
\pi^{-1}(\xi)\cong 
\overline{\M}_{0,|J_1|+1} \times \cdots \times \overline{\M}_{0,|J_l|+1}.
\eeq
The map $\pi$ is a birational equivalence. More precisely, let
$\Sigma\subset \PP^{n-2}$ be the analytic subvariety consisting of points
$\xi=(\xi_1:\cdots:\xi_n)$, such that, there exist 3 pairwise
different $i,j$, and $k$, such that, $\xi_i=\xi_j=\xi_k$. This is
clearly a complex co-dimension 2 algebraic subvariety and $\pi$ induces
an isomorphism $\pi^{-1}(\PP^{n-2}\setminus{\Sigma})\cong
\PP^{n-2}\setminus{\Sigma}$. Since $\PP^{n-2}$ is non-singular, the map
$\pi$ is in particular a rational resolution. The general theory of
rational singularities (see \cite{KM1998}, Theorem 5.10) yields  the
following lemma. 
\begin{lemma}\label{le:rat_sing}
The higher direct images of $\pi$ vanish, that is, $R^i\pi_* \O=0$ for $i>0$.
\end{lemma}

\subsection{J-function}
The $J$-function of a point was computed by Givental (see the Theorem in
\cite{Giv2015Aug}). We have the following formula:
\ben
J(\nu,q) =
1-q+\nu_1+\Big\langle \frac{1}{1-qL}\Big\rangle_{0,1}(\nu)= 
(1-q) \exp\Big(\sum_{k=1}^\infty \frac{\nu_k}{k(1-q^k)}\Big). 
\een
Let us outline a proof of the above formula which is slightly different than the one given in \cite{Giv2015Aug}.
Let $\op{ev}: \overline{\M}_{0,n+1}\to \op{pt}$ be the contraction map. Using Lemma \ref{le:rat_sing} we compute $\pi_* L_{n+1}^k = \pi_* \pi^* \O(k)= \O(k)\otimes \pi_*\O= \O(k)$. Therefore, 
\ben
\op{ev}_*(L_{n+1}^k)= H^0(\PP(V_n^*), \O(k))= \op{Sym}^k(V_n),\quad 
\op{Sym}^0(V_n):=\CC.
\een
The $J$-function takes the form
\ben
J(\nu,q)=1-q+\nu_1 + \sum_{n=2}^\infty \sum_{k=0}^\infty 
\op{Sym}^k(V_n) q^k.
\een
On the other hand, we have $\CC^n=V_n \oplus \CC$ as $S_n$-modules where $\CC$ is the trivial $S_n$-module. Note that $\op{Sym}^k(\CC^n)= \sum_{l=0}^k \op{Sym}^l(V_n)$ for $n\geq 1$ and $k\geq 0$ where we assume that $V_1=0$. Therefore, 
\ben
\sum_{n=1}^\infty \sum_{k=0}^\infty 
\op{Sym}^k(\CC^n) q^k= 
\sum_{k=0}^\infty \nu_1 q^k + 
\sum_{n=2}^\infty\sum_{k=0}^\infty 
\sum_{l=0}^k \op{Sym}^l(V_n) q^k=
\frac{\nu_1}{1- q} + 
\sum_{n=2}^\infty\sum_{l=0}^\infty 
\op{Sym}^l(V_n) q^l
\sum_{l=k}^\infty q^{k-l} ,
\een
where in the last equality we exchanged the summations over $k$ and $l$. Using the above formula we get that the $J$-function is
\ben
J(\nu,q)=(1-q)\Big(1 + \sum_{n=1}^\infty \sum_{k=0}^\infty 
\op{Sym}^k(\CC^n) q^k\Big).
\een
It remains only to prove that 
\ben
1 + \sum_{n=1}^\infty \sum_{k=0}^\infty 
\op{Sym}^k(\CC^n) q^k= \exp\Big(
\sum_{r=1}^\infty \frac{\nu_r}{r(1-q^r)}\Big).
\een
Let us denote the LHS of the above equality by $f(\nu,q)$. Let us compute the derivatives $r\partial_{\nu_r} f(\nu,q)$ by using the Heisenberg algebra structure of the Fock space. Recalling Theorem \ref{thm:Heisenberg}, we get 
\ben
r\partial_{\nu_r} \op{Sym}^k(\CC^n) q^k= 
\op{tr}_{(1,2,\dots,r)} \op{Res}^{S_n}_{S_r\times S_{n-r}}\, 
\op{Sym}^k(\CC^n) q^k=
\bigoplus_{a+b=k}
\op{tr}_{(1,2,\dots,r)}\op{Sym}^a(\CC^r) q^a \otimes 
\op{Sym}^{b}(\CC^{n-r}) q^b,
\een
where we take $n\geq r$, otherwise the derivative is $0$ for degree reasons. 
In order to compute this trace, let us identify $\op{Sym}^a(\CC^r)$ with the space of homogeneous polynomials in $\CC[x_1,\dots,x_r]$ of degree $a$. This space has a basis consisting of monomials $x_1^{d_1}x_2^{d_2}\cdots x_r^{d_r}$, such that, $d_1+d_2+\cdots +d_r=a$. Since the permutation $(1,2,\dots,r)$ permutes these monomials, the trace is equal to the number of monomials that remain fixed. However, there is at most one monomial fixed by  $(1,2,\dots,r)$, that is, the monomial with $d_1=\cdots=d_r$. Such monomial exists only if $r$ divides $a$. In other words, the trace in the above formula is $1$ if $r|a$ and $0$ otherwise. Summing over all $k$, we get 
\ben
r\partial_{\nu_r} \sum_{k=0}^\infty \op{Sym}^k(\CC^n) q^k= 
\frac{1}{1-q^r} \sum_{b=0}^\infty \op{Sym}^{b}(\CC^{n-r}) q^b.
\een
Summing over all $n\geq 1$, we get $r\partial_{\nu_r} f(\nu,q)=\frac{1}{1-q^r} f(\nu,q)$. Solving this differential equation with the initial condition $f(0,q)=1$ yields the formula that we had to prove. 

\subsection{Proof of Theorem \ref{thm:corr-0}}

Since $J(\nu,q)=(1-q)S(\nu,q)^{-1}1,$ we get 
\ben
S(\nu,q)=\exp\Big( 
\frac{\nu_1}{q-1} +\sum_{k=2}^\infty \frac{\nu_k}{k(q^k-1)}\Big)
\een
and 
\ben
G(\nu)= 1+\langle 1,1\rangle_{0,2}(\nu)= \partial_{\nu_1} J(\nu,0)= 
\exp\Big(
\nu_1+\sum_{k=2}^\infty \frac{\nu_k}{k}
\Big).
\een
Put
\ben
w(\mathbf{t}):=\sum_{n=0}^\infty \frac{1}{n!}
\langle 1,1 ,\mathbf{t},\dots,\mathbf{t}\rangle_{0,n+2}(0,\nu_2,\nu_3,\dots)
\een
and let $v(\mathbf{t})$ be defined as the solution to the equation 
\ben
w(\mathbf{t})= J(v(\mathbf{t}),0)-1=\exp\Big(
v(\mathbf{t}) + \sum_{k=2}^\infty \frac{\nu_k}{k}\Big)-1.
\een
Recalling Theorem \ref{thm:top_sol}, we get that $v(\mathbf{t})$ is a solution to the following system of differential equations:
\ben
\partial_n v = -\op{Res}_{q=\infty} dq (q-1)^n \partial \,
\exp\Big(
\frac{v}{q-1} +\sum_{k=2}^\infty \frac{\nu_k}{k(q^k-1)}\Big),
\een
where $\partial_n:=\tfrac{\partial }{\partial t_n}$ and $\partial=\partial_0$.
On the other hand, since
$\frac{1}{1-qL}=\frac{1}{1-q}\, 
\sum_{k=0}^\infty \frac{q^k}{(1-q)^k}\, (L-1)^k$, we get that the insertion of $\frac{1}{1-qL}$ is obtained from $w(\mathbf{t})$ via the action of the differential operator 
\ben
P(q)=\frac{1}{1-q}\, 
\sum_{k=0}^\infty \frac{q^k}{(1-q)^k}\, \partial_{k} = 
\frac{1}{1-q}\, 
\sum_{k=0}^\infty \frac{1}{(q^{-1}-1)^k}\, \partial_{k} ,
\een
that is, the correlator on the LHS in Theorem \ref{thm:corr-0} is $P(q_1)\cdots P(q_n) w(\mathbf{t})|_{t_0=\nu_1,t_1=t_2=\cdots = 0}$. 

Let us examine the action of $P(q)$ on 
\ben
w_a(\mathbf{t})= \exp\Big( a_1 v(\mathbf{t}) + \sum_{k=2}^\infty a_k \frac{\nu_k}{k}\Big)
\een
where $a=(a_1,a_2,a_3,\dots)$ is an arbitrary sequence of complex numbers. We have
\ben
(1-q)P(q) w_a(\mathbf{t})=\sum_{n=0}^\infty 
\frac{a_1}{(q^{-1}-1)^n}\, \partial_n v \, 
\exp\Big( a_1 v(\mathbf{t}) + \sum_{k=2}^\infty a_k \nu_k/k\Big).
\een
Recalling the differential equation for $v$, we get 
\ben
(1-q)P(q) w_a(\mathbf{t}) & = & 
- a_1
\exp\Big(a_1v(\mathbf{t}) + \sum_{k=2}^\infty a_k \nu_k/k\Big)\times\\
&&
\times
\sum_{n=0}^\infty 
\op{Res}_{u=\infty} du 
\frac{(u-1)^n}{(q^{-1}-1)^n}\, \partial \,
\exp\Big(\frac{v}{u-1} +\sum_{k=2}^\infty \frac{\nu_k}{k(u^k-1)}\Big).
\een
The residues in the above formula can be interpreted analytically as
follows. We expand the exponential as a formal power series in
$\mathbf{t},\nu_2,\nu_3,\dots$, then each coefficient is a rational
function in $u$ with poles on the unit circle. Therefore, we may and
we will think of each residue as $\tfrac{1}{2\pi
  \mathbf{i}}\oint_{|u-1|=R}$ where $R\gg 1$ is a sufficiently big
real number and the orientation of the contour is clockwise around the
center $1$. Let us choose $q$ so small that $|u-1|<|q^{-1}-1|$. The
sum over $n$ of the geometric series is uniformly convergent in $u$ to $\tfrac{1-q}{1-qu}$ and we get 
\ben
P(q) w_a(\mathbf{t})=- a_1
\exp\Big(a_1 v(\mathbf{t}) + \sum_{k=2}^\infty a_k \nu_k/k\Big)\,
\frac{\partial}{2\pi \mathbf{i}}\, \oint_{|u-1|=R} 
\exp\Big(\frac{v}{u-1} +\sum_{k=2}^\infty \frac{\nu_k}{k(u^k-1)}\Big)
\frac{du}{1-qu}.
\een 
Recalling the Cauchy residue theorem we get that the above integral is the sum of the residues at $u=q^{-1}$ and $u=\infty$. The residue at $u=\infty$ is a constant which does not contribute at the end because we have to apply to it the derivation $\partial$. We get 
\ben
P(q) w_a(\mathbf{t})=a_1 q^{-1}
\exp\Big(a_1 v(\mathbf{t}) + \sum_{k=2}^\infty a_k \nu_k/k\Big)\, \partial
\,
\exp\Big(\frac{v}{q^{-1}-1} +\sum_{k=2}^\infty \frac{\nu_k}{k(q^{-k}-1)}\Big).
\een
Note that the above formula is equivalent to 
\ben
P(q) w_a(\mathbf{t})=
\frac{1}{1-q}\, \frac{a_1}{a_1+\frac{1}{q^{-1}-1}} \, 
\partial\, 
\exp\left(\Big(a_1+\frac{1}{q^{-1}-1}\Big) v +
\sum_{k=2}^\infty  \Big(a_k+ \frac{1}{q^{-k}-1} \Big) \frac{\nu_k}{k}
\right).
\een
Using the above formula we get 
\ben
& & P(q_1)\cdots P(q_n) w(\mathbf{t}) = 
\frac{1}{(1-q_1)\cdots (1-q_n)} \, 
\frac{1}{1+\frac{1}{q_1^{-1}-1} + \cdots + \frac{1}{q_n^{-1}-1} }\times \\
&&
\times \partial^n\, \exp\left(
\Big(1+\frac{1}{q_1^{-1}-1}+\cdots + \frac{1}{q_n^{-1}-1} \Big) v(\mathbf{t}) +
\sum_{k=2}^\infty  
\Big(1+ \frac{1}{q_1^{-k}-1} +\cdots + \frac{1}{q_n^{-k}-1} \Big) \frac{\nu_k}{k}
\right).
\een
Note that if we set $t_1=t_2=\cdots =0$, then 
\ben
w(t_0,0,0,\dots)= G(t_0,\nu_2,\nu_3,\dots)-1= e^{t_0+\sum_{k=2}^\infty \frac{\nu_k}{k}}-1.
\een
Therefore, $v(t_0,0,0,\dots)=t_0$ and the formula that we had to prove follows.
\qed


\bibliographystyle{plain}
\bibliography{fock-arXiv}

\end{document}